%% file: mayer_vietoris.tex
\documentclass[a4paper]{amsart}
\input{preamble}

\usepackage{xcolor}
\newcommand{\tfF}{\mathsf{F}} % general 2-functor
\begin{document}
\begin{abstract}
  \input{mayer_vietoris_abstract}
\end{abstract}
\maketitle

\section{Introduction}
Let $X$ be a topological space and let $U$, $V \subseteq X$ be open subsets such 
that $X=U\cup V$. Frequently, natural invariants of $X$ can be determined by the 
restriction of these invariants to $U$, $V$ and $U \cap V$. The prototypical 
examples are the Mayer--Vietoris exact sequences in algebraic topology. These 
results have proved to be very useful for inductive arguments.

Now let $X$ be a variety, scheme, algebraic space, or algebraic stack. It is 
straightforward to adapt the topological results (e.g., Mayer--Vietoris with open coverings) to this situation. In algebraic geometry, however, open coverings are often too restrictive to use in inductive arguments. A consideration of the existing literature motivated us to make the following definition.
\begin{definition}
  Consider a cartesian diagram of algebraic stacks
  \begin{equation}\label{E:MV-square}
    \vcenter{\xymatrix{
        U'\ar[r]^{j'}\ar[d]_{f_U} & X'\ar[d]^f \\
        U\ar[r]^j & X,\ar@{}[ul]|\square
      }}
  \end{equation}
  where $j$ is an open immersion. It is a \emph{weak Mayer--Vietoris square}
  if for every morphism of algebraic stacks $W \to X$ with image disjoint
  from $U$, the induced morphism $f_W \colon W'=X'\times_X W \to W$ is an
  isomorphism. It will be convenient 
  to let $i\colon Z \inj X$ denote a closed immersion with complement $U$.
  Then $Z':=f^{-1}(Z)\to Z$ is an isomorphism.
\end{definition}
The condition of being a weak Mayer--Vietoris square is trying to capture that $X'$ contains 
all infinitesimal neighborhoods of $Z$ in $X$. In particular, if $X$ and $X'$ are locally 
noetherian, then being a weak Mayer--Vietoris square is equivalent to $f_Z$ being an isomorphism and $f$ being flat at all 
points over $Z$ (Lemma \ref{L:MV-different-notions}). 

Fix a weak Mayer--Vietoris square as in \eqref{E:MV-square}.  
If $f$ is \'etale, then it is also known as an \emph{\'etale
neighborhood}, or upper distinguished square, or Nisnevich square. These were treated in depth in~\cite{MR2774654}. Some highlights of the theory are that \'etale neighborhoods are pushouts in the $2$-category of algebraic stacks and that quasi-coherent sheaves (and many more things) can be glued along \'etale neighborhoods. \'Etale neighborhoods feature 
prominently in the interactions between algebraic geometry and topology.

There is another class of weak Mayer--Vietoris squares, which generalize \'etale neighborhoods, that have been 
considered and applied to great effect in the past \cite{MR0272779,MR1374653}. These are our main object of interest.
\begin{definition}
  A \emph{flat Mayer--Vietoris square} is a weak Mayer--Vietoris square as in \eqref{E:MV-square} such that $f$ is flat.
\end{definition}
Our first main result is the following.
\begin{maintheorem}\label{MT:pushout_fmv}
  Fix a flat Mayer--Vietoris square as in \eqref{E:MV-square}. If $X$ is
  locally the spectrum of a $G$-ring (e.g., locally excellent), then the square
  \eqref{E:MV-square} is a pushout in the $2$-category of algebraic stacks.
\end{maintheorem}

For a discussion of $G$-rings, see \cite[Tag \spref{07GG}]{stacks-project}. An algebraic stack is locally excellent if it admits a smooth cover by an excellent scheme; in particular, algebraic stacks that are locally of finite type over the spectrum of a field, $\Z$, or a complete local noetherian ring are locally excellent. 
Theorem \ref{MT:pushout_fmv} is the key technical result used to establish Tannaka 
duality for algebraic stacks with non-separated diagonals 
\cite{hallj_dary_coherent_tannakian_duality}---if the diagonals are separated, then 
\cite[Cor.~6.5.1(g)]{MR1432058} is sufficient for the Tannakian application.

In general, Mayer--Vietoris squares are interesting since objects can be glued along
them. To formalize this, consider a $2$-presheaf $\tfF\colon
  (\STACKS{X})^\opp\to \CAT$~\cite[App.~D]{MR2774654}, e.g., the $2$-presheaf
  $\tfF(-)=\QCOH(-)$ of quasi-coherent sheaves of modules. By pull-back, we
  obtain a functor
  \[
  \Phi_\tfF\colon \tfF(X)\to  \tfF(X')\times_{\tfF(U')} \tfF(U) 
  \]
  where the right-hand side denotes triples $(W',\theta,W_U)$ where
  $W'\in \tfF(X')$, $W_U\in \tfF(U)$ and
  $\theta\colon j'^*W'\to f_U^*{W_U}$ is an
  isomorphism. When the functor $\Phi_\tfF$ is an equivalence, we
  say that we can glue $\tfF$ along the square. 

  We do not have any general gluing results for $2$-\emph{sheaves} as for
  \'etale neighborhoods~\cite[Thm.~A]{MR2774654}---nor do we expect such---but we
  will give gluing results for the following $2$-presheaves, where the values over an algebraic stack $Y$ are as follows:
  \begin{itemize}[leftmargin=2cm]
    \item[$\QCOH(Y)$] the category of quasi-coherent sheaves of $\Orb_Y$-modules.
    % \item[$\QCOH_{f-\mathrm{fl}}(Y)$] the category of $f$-flat quasi-coherent sheaves of $\Orb_Y$-modules (Definition~\ref{D:f-flat}).
    \item[$\AFF(Y)$] the category of affine morphisms $Y'\to Y$.
    \item[$\QAFF(Y)$] the category of quasi-affine morphisms $Y'\to Y$.
    % \item[$\AFF_{f-\mathrm{fl}}(Y)$] the category of $f$-flat affine morphisms to $Y$;
    \item[$\AlgSp(Y)$] the category of representable morphisms $Y'\to Y$.
    \item[$\AlgSp_\lfp(Y)$] the category of representable morphisms $Y'\to Y$,
      locally of finite presentation.
    \item[$\Hom(Y,W)$] the groupoid of morphisms $Y\to W$ to a fixed algebraic stack $W$.
    % \item[$\AlgSp_{f-\mathrm{fl}}(Y)$] the category of $f$-flat representable morphisms $Y'\to Y$ (Definition~\ref{D:f-flat}).
    \item[$\Et(Y)$] the category of \'etale representable morphisms $Y'\to Y$,
      or equivalently, the category of cartesian sheaves
      of sets on the lisse-\'etale site of $Y$.
    \item[$\Et_c(Y)$] the category of finitely presented \'etale representable
      morphisms $Y'\to Y$, or equivalently, the category of constructible
      sheaves of sets.
  \end{itemize}
  We prove the following gluing results. 
  \begin{maintheorem}\label{MT:glue_fmv}
    Fix a flat Mayer--Vietoris square as in \eqref{E:MV-square}. If $j$ is quasi-compact, 
    then
    \begin{enumerate}
    \item\label{MTI:glue_fmv:qcoh+aff+qaff}
      $\Phi_{\QCOH}$, $\Phi_{\AFF}$ and $\Phi_{\QAFF}$ are
      equivalences of categories;
    \item\label{MTI:glue_fmv:algsp-ff}
      $\Phi_{\AlgSp}$ is fully faithful;
    \item\label{MTI:glue_fmv:hom-ff}
      $\Phi_{\Hom(-,W)}$ is fully faithful for every algebraic stack $W$ and an 
      equivalence if $W$ has quasi-affine diagonal; and
    \item\label{MTI:glue_fmv:algsp-exc}
      $\Phi_{\AlgSp_{\lfp}}$ is an equivalence of categories if $X$ is
      locally the spectrum of a $G$-ring.
    \end{enumerate}
  \end{maintheorem}
  \begin{maintheorem}\label{MT:etale-gluing-for-wmv}
    Fix a weak Mayer--Vietoris square as in \eqref{E:MV-square}. 
    If $j$ is quasi-compact, then $\Phi_{\Et}$ and $\Phi_{\Et_c}$ are equivalences of categories.
  \end{maintheorem}
  Theorem \ref{MT:pushout_fmv} essentially follows from Theorem \ref{MT:glue_fmv}\itemref{MTI:glue_fmv:algsp-exc}.
  Theorem \ref{MT:glue_fmv}\itemref{MTI:glue_fmv:algsp-ff}--\itemref{MTI:glue_fmv:algsp-exc} relies upon Theorem~\ref{MT:etale-gluing-for-wmv}
  and general N\'eron--Popescu desingularization \cite{MR818160}.
  Gabber's rigidity
  theorem (Theorem~\ref{T:rigidity-theorem}) features in the proof of 
  Theorem~\ref{MT:etale-gluing-for-wmv}. Both results also depend upon some further 
  gluing results for quasi-coherent sheaves.

  To prove Theorem \ref{MT:glue_fmv}, we approximately follow the approach of 
  \cite{MR1432058}. The main idea is pass to a square as in 
\eqref{E:MV-square} where $f\colon X' \to X$ is replaced by its diagonal $\Delta_f 
\colon X' \to X'\times_X X'$ and $j\colon U \to X$ is replaced by $j'\times j' \colon 
U'\times_{X} U' \to X'\times_X X'$. In particular, unless $f$ is unramified the resulting square will not be a flat Mayer--Vietoris square. Moreover, even if $X$ is a locally noetherian algebraic stack, then unless $f$ is locally of finite type, $X'\times_X X'$ has no reason to be locally noetherian. For example, in applications one often takes $X=\spec A$ and $X'=\spec \hat{A}$, where $\hat{A}$ denotes the $I$-adic completion with respect to some ideal $I$ of $A$; in this situation, $X'\times_X X'$ is only noetherian when $X'=X$.

To manage such squares, we have the following natural variant of what Moret-Bailly considered.
\begin{definition}
  A \emph{tor-independent Mayer--Vietoris square} is a weak Mayer--Vietoris square 
  as in \eqref{E:MV-square} such that every morphism of algebraic stacks $W \to X$ with image disjoint from $U$ is tor-independent of $f$ (Definition \ref{D:f-flat}).
\end{definition}
In other words, the
\emph{derived} base change $X'\overset{\RDERF}{\times_X} W\to W$ is an
isomorphism for every $W\to X$ with image disjoint from $U$.

If $f$ is affine and $U$ is the complement of a finitely presented closed immersion $i\colon Z \inj X$, then a tor-independent Mayer--Vietoris square is the same as a triple $(X,Z,X')$ satisfying the (TI) condition in the terminology of~\cite[0.2, 0.6]{MR1432058} (Lemma \ref{L:MV-open-qc}\itemref{LI:MV-open-qc:ti}). If $X'$ and $X$ are locally noetherian, then tor-independent Mayer--Vietoris squares are very similar to flat Mayer--Vietoris squares (Lemma \ref{L:MV-different-notions}).  % While flat and weak Mayer--Vietoris squares are stable under arbitrary base change (Lemma \ref{L:mv_bc}\itemref{LI:mv_bc:w_f_mv}), tor-independent Mayer--Vietoris square are only stable under flat base change (Example \ref{E:mv-not-umv}). In particular, Theorem \ref{MT:etale-gluing-for-wmv} is much easier to prove in the setting of weak Mayer--Vietoris squares. 
We now state our gluing result for tor-independent Mayer--Vietoris squares, which we can prove for $f$-flat objects (see Definition \ref{D:f-flat}).
\begin{maintheorem}\label{MT:glue_timv}
  Fix a tor-independent Mayer--Vietoris square as in \eqref{E:MV-square}. If $j$ is 
  quasi-compact, then 
  \begin{enumerate}
  \item \label{MTI:glue_timv:qcoh}$\Phi_{\QCOH_{f-\mathrm{fl}}}$ is an equivalence and
  \item \label{MTI:glue_timv:algsp}$\Phi_{\AlgSp_{f-\mathrm{fl}}}$ is fully faithful.
  \end{enumerate}
\end{maintheorem}
For tor-independent Mayer--Vietoris squares, we prove the following non-noetherian variant of Theorem \ref{MT:pushout_fmv}. 
\begin{maintheorem}\label{MT:timv_dm_push}
  Fix a tor-independent Mayer--Vietoris square as in \eqref{E:MV-square}. If $j$ 
  is quasi-compact, then it is a pushout in the $2$-category of Deligne--Mumford 
  stacks.
\end{maintheorem}
Since we make no separation assumptions on our algebraic stacks, Theorem 
\ref{MT:timv_dm_push} generalizes recent work of Bhatt 
\cite{2014arXiv1404.7483B}. Note, however, that while Bhatt uses (derived) 
Tannaka duality to prove a version of Theorem \ref{MT:timv_dm_push} for quasi-compact and quasi-separated algebraic spaces, we work in the opposite direction (i.e., we use pushouts to prove Tannaka duality in~\cite{hallj_dary_coherent_tannakian_duality}). 

\begin{remark}
  While it may appear that our results are weaker than the
  corresponding \'etale gluing results~\cite{MR2774654}
  because we require $j$ to be
  quasi-compact, this turns out to not be the case. Indeed,
  smooth-locally on $X$ there is an \'etale 
  neighborhood $X''$ of $Z'$ in $X'$ such that $X''\to X$ is quasi-affine
  (Proposition~\ref{P:qaff_dom}). If $f\colon X'\to X$ is an \'etale
  neighborhood, then $X''\to X$ is of finite presentation so we can
  find an open $U_0 \subseteq 
  U$ such that $U_0 \to X$ is quasi-compact and the resulting square with $X'' \to X$ is an 
  \'etale neighborhood of $X\smallsetminus U_0$.
\end{remark}

\subsection*{Overview}
  In Section~\ref{S:prelim} we give some preliminaries on tor-independence. In
  Section~\ref{S:MV} we compare the different notions of Mayer--Vietoris
  squares and give several examples.  In Section~\ref{S:QCoh-gluing} we glue
  quasi-coherent sheaves in tor-independent Mayer--Vietoris squares
  (Theorem~\ref{MT:glue_timv}\itemref{MTI:glue_timv:qcoh} and
  Theorem~\ref{MT:glue_fmv}\itemref{MTI:glue_fmv:qcoh+aff+qaff}).

  In Section~\ref{S:etale-sheaves} we prove some fundamental theorems for
  \'etale sheaves of sets on algebraic stacks. In particular, we prove that
  every sheaf on a quasi-compact and quasi-separated algebraic stack is a
  filtered colimit of constructible sheaves. We also discuss henselian pairs of stacks.

  In Section~\ref{S:etale-gluing} we prove Gabber's rigidity theorem and glue
  \'etale sheaves in weak Mayer--Vietoris squares
  (Theorem~\ref{MT:etale-gluing-for-wmv}).
  In the noetherian case, Gabber's rigidity theorem follows immediately from
  Ferrand--Raynaud~\cite[App.]{MR0272779}. In the non-noetherian case, which is
  essential for the applications in this paper, the previous
  proof~\cite[Exp.~20]{MR3309086} was much more involved. Using our results on
  gluing of sheaves, we provide a self-contained proof (for $H^0$ but the methods
  can be extended to $H^1$).

  In Section~\ref{S:algsp-gluing} we glue algebraic spaces and prove that
  Mayer--Vietoris squares are pushouts (Theorems~\ref{MT:pushout_fmv},
  \ref{MT:glue_fmv}, \ref{MT:glue_timv}\itemref{MTI:glue_timv:algsp},
  and~\ref{MT:timv_dm_push}).

  % To prove Theorem \ref{MT:glue_timv}\itemref{MTI:glue_timv:algsp}, we use 
  % Theorem \ref{MT:glue_timv}\itemref{MTI:glue_timv:qcoh} and a gluing result for \'etale 
  % sheaves. However, since 
  % tor-independent Mayer--Vietoris squares are not compatible with arbitrary base 
  % change (Example \ref{E:mv-not-umv}), this makes proving such a result awkward. For 
  % this reason, we found it convenient to introduce the notion of \emph{weak 
  %   Mayer--Vietoris square} (Definition \ref{D:wmv}).  Note that every tor-independent Mayer--Vietoris square is a weak Mayer--Vietoris 
  %   square (Lemma \ref{L:MV-different-notions}).
  \subsection*{Acknowledgements}
  We would like to thank an anonymous referee for a number of helpful comments and suggestions. 
\section{Preliminaries}\label{S:prelim}
Here we record some preliminary results that will be of use in subsequent sections. 
Most of these are globalizations of the affine results proved in 
\cite[\S2]{MR1432058}. We begin with the following definition.
\begin{definition}\label{D:f-flat}
  Let $f\colon X' \to X$ and $g\colon W \to X$ be morphisms of algebraic stacks.  Let $N 
  \in \QCOH(X')$ and $M \in \QCOH(W)$.
  \begin{enumerate}
  \item We say that $M$ and $N$ are \emph{tor-independent} if 
    $\mathscript{T}or_i^{X,f,g}(N,M) = 0$
    for all $i>0$ \cite[App.~C]{MR3589351}. Equivalently, for all smooth 
    morphisms $\spec A \to X$, $\spec A' \to \spec A \times_X X'$ and $\spec B \to \spec 
    A\times_X W$ we have 
    \[
    \Tor^A_i(N(\spec A' \to X'),M(\spec B \to W)) = 0
    \]
    for all $i>0$. 
  \item We say that $M$ is \emph{$f$-flat} if it is tor-independent of $\Orb_{X'}$.  We let 
    $\QCOH_{f-\mathrm{fl}}(W) \subseteq \QCOH(W)$ denote the subcategory of $f$-flat 
    quasi-coherent sheaves on $W$.
  \item We say that $g$ is \emph{$f$-flat} if $\Orb_W$ is $f$-flat. Note that $g$ is $f$-flat 
    if and only if $f$ is $g$-flat. In particular, we may also say that $f$ and $g$ are 
    tor-independent.
  \end{enumerate}
\end{definition}
The following lemma is immediate from the definitions (we employ the notational conventions from \cite{perfect_complexes_stacks}).
\begin{lemma}\label{L:der_char_fflat}
  Let $f\colon X' \to X$ be a morphism of algebraic stacks and let $M \in 
  \QCOH(X)$. Then $M$ is $f$-flat if and only if the natural map $\LDERF \QCPBK{f}M \to 
  f^*M$ is a quasi-isomorphism in $\DQCOH(X')$.
\end{lemma}
The following notation will also be useful.
\begin{notation}
  Let $i\colon Z \inj X$ be a closed immersion of algebraic stacks, which is defined by the
  quasi-coherent ideal $I$. For each integer $n\geq 0$, let $i^{[n]} \colon Z^{[n]} \inj X$ be 
  the closed immersion defined by the quasi-coherent ideal $I^{n+1}$. Note that if $i$ is a 
  finitely presented closed immersion, then so too is $i^{[n]}$ for all $n\geq 0$.
\end{notation}
The following lemma will eventually be improved (see Corollary \ref{C:gluing_special}), but is for the meantime sufficient for our purposes.
\begin{lemma}\label{L:flatness_conc_cl}
  Consider a cartesian diagram of algebraic stacks:
  \[
  \xymatrix{V' \ar[r]^{v'} \ar[d]_{g_V} & W' \ar[d]^g\\V \ar[r]^v & W. }
  \]
  Assume that $v$ is $g$-flat and a closed immersion.
  \begin{enumerate}
  \item \label{LI:flatness_conc_cl:conc}  If $M \in 
    \QCOH_{g_v-\mathrm{fl}}(V)$, then $v_*M \in \QCOH_{g-\mathrm{fl}}(W)$.
  \item \label{LI:flatness_conc_cl:iso} If $g_V$ is 
    an isomorphism, then $V^{[n]}$ is $g$-flat and $g_{V^{[n]}}  \colon 
    W'\times_W V^{[n]} \to V^{[n]}$ is an isomorphism  for every 
    $n\geq 0$.
  \end{enumerate}
\end{lemma}
\begin{proof}
  Both claims are smooth local on $W$, so we may assume that $W=\spec A$ is an 
  affine scheme. Claim \itemref{LI:flatness_conc_cl:conc} follows from tor-independent 
  base change \cite[Cor.~4.13]{perfect_complexes_stacks}. Indeed, this provides 
  quasi-isomorphisms: 
  \[
  \LDERF \QCPBK{g}(v_*M) \simeq \LDERF \QCPBK{g}\RDERF v_*M \simeq \RDERF v'_*
  \LDERF\QCPBK{(g_V)}M \simeq v'_*g_V^*M \simeq g^*v_*M.
  \]  
  Claim \itemref{LI:flatness_conc_cl:iso} is essentially the local criterion for 
  flatness \cite[$0_{\mathrm{III}}$.10]{EGA}, but we will spell out the details. Assume that $V=\spec (A/I)$. 
  Fix an integer $n\geq 1$. By induction we may also assume 
  that $g_{V^{[n-1]}} \colon W'\times_W \spec (A/I^n) \to \spec (A/I^n)$ is an isomorphism
  and $A/I^n$ is $g$-flat. Now \itemref{LI:flatness_conc_cl:conc} applied to $V^{[n-1]} \to W$ implies that
  every 
  $A/I^n$-module is $g$-flat. In particular, $I^n/I^{n+1}$ is $g$-flat so the distinguished 
  triangle
  \[
  \LDERF\QCPBK{g}(I^n/I^{n+1}) \to \LDERF \QCPBK{g}(A/I^{n+1}) \to \LDERF \QCPBK{g}(A/I^{n}) \to   \LDERF\QCPBK{g}(I^n/I^{n+1})[1]
  \]
  now implies that $A/I^{n+1}$ is $g$-flat. Since $V'\cong V$ is affine, so is $V'^{[n]}$.
  The $0$th cohomology of the distinguished triangle above fits in the exact sequences
  \[
  \xymatrix@C-2mm{%
  0\ar[r] & I^n/I^{n+1}\ar[r]\ar[d]^{\cong} & A/I^{n+1}\ar[r]\ar[d] & A/I^n\ar[r]\ar[d]^{\cong} & 0 \\
  0\ar[r] & \Gamma(V',g^*I^n/I^{n+1})\ar[r] & \Gamma(V'^{[n]},\Orb_{V'^{[n]}})\ar[r] & \Gamma(V'^{[n-1]},\Orb_{V'^{[n-1]}})\ar[r] & 0}
  \]
  so $g_{V^{[n]}}$ is an isomorphism.
\end{proof}
Let $X$ be an algebraic stack and let $i\colon Z \inj X$ be a closed immersion with complement $j\colon U \to X$. Define
\[
\QCOH_{Z}(X) = \{M \in \QCOH(X) \suchthat j^*M \cong 0\}.
\]
Note that $\QCOH_{Z}(X)$ only depends on the closed subset $|Z| \subseteq |X|$.
\begin{lemma}\label{L:fp_thickenings}
  Let $X$ be a quasi-compact algebraic stack. Let $i\colon Z \inj X$ be a finitely 
  presented closed immersion.
  \begin{enumerate}
  \item \label{LI:fp_thickenings:ft} Let $M \in \QCOH_{Z}(X)$. If $M$ is of finite type, then there exists an $n\gg 0$ such that the natural map 
      $M \to i_*^{[n]}(i^{[n]})^*M$ is an isomorphism. 
  \item\label{LI:fp_thickenings:cl} If $W \subset X$ is a closed substack with $|W| \subseteq |Z|$, then $W \subseteq Z^{[n]}$ for some~${n\gg 0}$.
  \end{enumerate}
\end{lemma}
\begin{proof}
  For \itemref{LI:fp_thickenings:ft}: we may assume that $X=\spec A$ is an affine scheme 
  and $Z=\spec (A/I)$, where $I=(f_1,\dots,f_r)$ is a finitely generated ideal of $A$. By 
  assumption, $M_{f_i} = 0$ for each $i=1$, $\dots$, $r$. As $M$ is finitely generated, it 
  follows that there exists $n\gg 0$ such that $f_i^nM = 0$ for all $i=1$, $\dots$, $r$. The 
  claim follows.

  For \itemref{LI:fp_thickenings:cl}: let $W_0 = Z\times_X W$. Then  $W_0 \inj W$ is a 
  surjective and finitely 
  presented closed immersion. From \itemref{LI:fp_thickenings:ft}, it follows that $W 
  \subseteq W_0^{[n]}$ for some $n\gg 0$. But $W_0^{[n]} \subseteq Z^{[n]}$ and we have 
  the claim. 
\end{proof}
\section{Mayer--Vietoris squares}\label{S:MV}
In this section, we compare various notions of Mayer--Vietoris squares. 
\begin{lemma}\label{L:mv_bc}
  Fix a cartesian diagram as in \eqref{E:MV-square}.
  \begin{enumerate}
  \item \label{LI:mv_bc:w_f_mv} If the square is a weak (resp.~flat) Mayer--Vietoris 
    square, then it remains so after arbitrary base change on $X$.
  \item \label{LI:mv_bc:mv} If the square is a tor-independent Mayer--Vietoris square, then it 
    remains so after $f$-flat base change on $X$.
  \item \label{LI:mv_bc:verify}  The properties of being a flat, tor-independent, or weak 
    Mayer--Vietoris square are flat local on $X$.
  \end{enumerate}
\end{lemma}
\begin{proof}
  Claim \itemref{LI:mv_bc:w_f_mv} is trivial. For \itemref{LI:mv_bc:mv}: let
  $v\colon V \to X$ be $f$-flat and let $w\colon W \to V$ be such that
  the image of $w$ is disjoint from $v^{-1}(U)$. Then $v\circ w$ has image
  disjoint from $U$ so is $f$-flat. It follows that $w$ is $f_V$-flat where
  {$f_V \colon X'\times_X V \to V$}.
  Indeed, this is local on $W$, $V$, $X$ and $X'$, so we may assume that
  $X=\spec A$, $X'=\spec A'$, $V=\spec B$ and $W=\spec C$ and then
  \[
  C\tensor^{\LDERF}_B (B\tensor_A A') \simeq C\tensor^{\LDERF}_A A'
  \simeq C\tensor_A A' \simeq C\tensor_B (B\tensor_A A')
  \]
  since $v$ and $v\circ w$ are $f$-flat.
  The claim \itemref{LI:mv_bc:verify} is immediate from flat 
  descent.
\end{proof}
As the following Lemma shows, the conditions for Mayer--Vietoris squares are
much easier to check when a description of the complement is given.
\begin{lemma}\label{L:MV-open-qc}
  Fix a cartesian diagram as in \eqref{E:MV-square}. Suppose that $U$ is the complement 
  of a finitely presented closed immersion $i\colon Z \inj X$.
  \begin{enumerate}
  \item \label{LI:MV-open-qc:weak} If $f_{Z^{[n]}} \colon X'\times_X Z^{[n]} \to 
    Z^{[n]}$ is an isomorphism for all $n$, then the square is a weak 
    Mayer--Vietoris square.
  \item \label{LI:MV-open-qc:ti} If $f_{Z} \colon X'\times_X Z \to Z$ is an isomorphism and 
    $i$ and $f$ are tor-independent, then the square is a tor-independent Mayer--Vietoris 
    square. 
  \end{enumerate}
\end{lemma}
\begin{proof}
  We may assume that $X$ is affine (Lemma \ref{L:mv_bc}\itemref{LI:mv_bc:verify}). 
  Let $g\colon W \to X$ be a morphism of algebraic stacks with 
  image disjoint from $U$. We must prove that $f_W \colon X'\times_X W \to W$ is an 
  isomorphism and for \itemref{LI:flatness_conc_cl:iso} also that $f$ and $g$
  are tor-independent. These claims are smooth local on $W$, so we may also assume that $W$ is affine. 
  The morphism $g$ is now affine, so its schematic image $V$ exists and is disjoint from 
  $U$. In particular, $|V| \subseteq |Z|$. By Lemma 
  \ref{L:fp_thickenings}\itemref{LI:fp_thickenings:cl}, $V \subseteq Z^{[n]}$ for some $n\gg 
  0$. Hence, $W \to X$ factors through $Z^{[n]}$ for some $n\gg 0$. The claim 
  \itemref{LI:MV-open-qc:weak} is now immediate.  For \itemref{LI:MV-open-qc:ti}, the 
  result follows from Lemma \ref{L:flatness_conc_cl}\itemref{LI:flatness_conc_cl:iso}.
\end{proof}
Note that if $X$ is quasi-compact and quasi-separated and $j$ is quasi-compact, then $i\colon Z \inj X$ as in Lemma \ref{L:MV-open-qc} always exists \cite[Prop.~8.2]{rydh-2014}.

The following lemma connects the various types of Mayer--Vietoris squares to each other.
\begin{lemma}\label{L:MV-different-notions}
  Fix a square as in \eqref{E:MV-square}. Consider the following 
  conditions. 
  \begin{enumerate}
  \item \label{LI:MV-different-notions:fmv} The square is a flat Mayer--Vietoris 
    square.
  \item \label{LI:MV-different-notions:mv_fz} The square is a weak Mayer--Vietoris square and $f$ is flat at every point of $Z'$.
  \item \label{LI:MV-different-notions:mv} The square is a tor-independent 
    Mayer--Vietoris square.
  \item \label{LI:MV-different-notions:wmv} The square is a weak Mayer--Vietoris square.
  \end{enumerate}
  Then 
  \itemref{LI:MV-different-notions:fmv}$\implies$\itemref{LI:MV-different-notions:mv_fz}$\implies$\itemref{LI:MV-different-notions:mv}$\implies$\itemref{LI:MV-different-notions:wmv}.
  If $X$ and $X'$ are locally noetherian, then~\itemref{LI:MV-different-notions:wmv}$\implies$\itemref{LI:MV-different-notions:mv_fz}. If there exists a Cartier divisor 
    $i\colon Z \inj X$ with complement $U$ such that $f^{-1}(Z)\to Z$ is an isomorphism
    and $f^{-1}(Z) \inj X'$ is 
    also a Cartier divisor, then~\itemref{LI:MV-different-notions:mv} holds.
  \end{lemma}
  \begin{proof}
    That
    \itemref{LI:MV-different-notions:fmv}$\implies$\itemref{LI:MV-different-notions:mv_fz}$\implies$\itemref{LI:MV-different-notions:mv}$\implies$\itemref{LI:MV-different-notions:wmv}
    is obvious. If $X$ and $X'$ are locally noetherian,
    then~\itemref{LI:MV-different-notions:mv_fz} follows from the local
    criterion of flatness~\cite[0$_{\mathrm{III}}$.10.2.1--2]{EGA} (the
    conditions are flat-local on $X$ and $X'$ so reduces to schemes).

    For the last claim, it is sufficient 
    to prove that $\Orb_Z$ is $f$-flat (Lemma \ref{L:MV-open-qc}\itemref{LI:MV-open-qc:ti}), which is local on $X$. So we may assume that 
    $Z=V(s)$ and obtain an exact sequence
    \[
    0\to \Orb_X\xrightarrow{s\cdot} \Orb_X\to \Orb_Z\to 0.
    \]
    Applying $\LDERF \QCPBK{f}$ to this, we obtain a distinguished triangle in $\DQCOH(X')$:
    \[
    \Orb_{X'} \to \Orb_{X'} \to \LDERF \QCPBK{f}\Orb_Z \to \Orb_{X'}[1].
    \]
    The resulting long exact sequence of cohomology yields:
    \[
    0\to \COHO{-1}(\LDERF\QCPBK{f}\Orb_Z) \to \Orb_{X'}\xrightarrow{s\cdot} \Orb_{X'}\to f^*\Orb_{Z}\to 0,
    \]
    with all other terms $0$. Since $f^{-1}(Z)\inj X'$ is a Cartier divisor, $s$ is regular on 
$\Orb_{X'}$ and so $\COHO{-1}(\LDERF\QCPBK{f}\Orb_Z)=0$. 
Hence, $\LDERF \QCPBK{f}\Orb_Z \to f^*\Orb_Z$ is a quasi-isomorphism and the 
square is a tor-independent Mayer--Vietoris square.
  \end{proof}
As the following lemma shows, blowing up provides a natural way to move from the weak Mayer--Vietoris setting to the tor-independent setting.
  \begin{lemma}\label{L:blow-up-is-MV-square}
    Fix a weak Mayer--Vietoris square as in \eqref{E:MV-square}. If there is a finitely presented closed immersion $i\colon Z \inj X$ with complement $U$, then
  \[
  \vcenter{\xymatrix{
      U'\ar[r]^-{j'}\ar[d]_{f_U} & \Bl_{Z'}X'\ar[d]^{\tilde{f}} \\
      U\ar[r]^-{j} & \Bl_Z X\ar@{}[ul]|\square
    }}
  \] 
  is a tor-independent  Mayer--Vietoris square.
  \end{lemma}
  \begin{proof}
    Since the exceptional divisors $E\inj \Bl_Z X$ and $E'\inj \Bl_{Z'} X'$ are
    Cartier divisors it is enough to verify that $E'\to E$ is an isomorphism
    (Lemma \ref{L:MV-different-notions}). Let $I$ be the ideal defining $Z\inj X$ and $I'$
    the ideal defining $Z'\inj X'$. Then the inverse images of $Z$ in the
    two blow-ups are
    \[
    E=\mathrm{Proj}_X(\oplus_{k\geq 0} I^k/I^{k+1})\quad\text{and}\quad
    E'=\mathrm{Proj}_{X'}(\oplus_{k\geq 0} I'^k/I'^{k+1}).
    \]
    Since $I'=I\Orb_{X'}$ and $\Orb_X/I^m\to \Orb_{X'}/I'^m$ is an isomorphism
    for every $m$, these two graded rings are isomorphic
    $\Orb_X/I=\Orb_{X'}/I'$-algebras. The result follows.
  \end{proof}
  The following lemma is a key observation of Moret-Bailly and will be essential to the article.
  \begin{lemma}[{\cite[Cor.~2.5.1]{MR1432058}}]\label{L:sections-give-MV}
    Fix a weak Mayer--Vietoris square as in \eqref{E:MV-square}.
    If $f$ admits a section $s\colon X\to X'$, then 
    \[
    \vcenter{\xymatrix{
      U\ar[r]^{j}\ar[d]_{s_U} & X \ar[d]^s\\
      U'\ar[r]^{j'} & X'\ar@{}[ul]|\square
    }}
    \]
    is a weak Mayer--Vietoris square. Moreover, if the square \eqref{E:MV-square} is a 
    tor-independent Mayer--Vietoris square, then so too is the one above.
  \end{lemma}
  \begin{proof}
    Let $w'\colon W' \to X'$ be a morphism with image disjoint from $U'$. It follows that the 
    composition $W' \xrightarrow{w'} X' \xrightarrow{f} X$ has image disjoint from $U$ 
    and so $W'\times_{X} X' \to 
    W'$ is an isomorphism. In particular, the following diagram is cartesian:
    \[
    \xymatrix{X \ar[d]_s & \ar[l]_{f\circ w'} W' \ar@{=}[d]\\X' \ar[d]_f & \ar[l]_{w'} W' \ar@{=}[d]\\ X & \ar[l]_{f\circ w'} W' }
    \]
    and so $W'\times_{X'} X \to W'$ is an isomorphism 
    as required. If the square \eqref{E:MV-square} is a tor-independent Mayer--Vietoris
    square, then the lower square in the diagram above is tor-independent.
    Since the whole square is tor-independent, it follows that the upper
    square is tor-independent and the last claim follows.
  \end{proof}
  \begin{example}[{\cite[Prop.~2.5.2]{MR1432058}}]\label{E:diagonal-square-is-MV}
    Fix a weak Mayer--Vietoris square as in \eqref{E:MV-square}. Then
     \[
    \xymatrix@C+5mm{
      U' \ar[r]^{j'}\ar[d]_{\Delta_{f_U}} & X' \ar[d]^{\Delta_f}\\
      U'\times_U U' \ar[r]^{j'\times j'} & X'\times_X X'\ar@{}[ul]|\square
    }
    \]
    is a weak Mayer--Vietoris square. Indeed, we can base change the square 
    \eqref{E:MV-square} by $X' \to X$ and the resulting square is still weak
    (Lemma \ref{L:mv_bc}). Taking the diagonal section to the 
    projection $X'\times_X X' \to X'$ and using Lemma~\ref{L:sections-give-MV} gives the 
    claim. If the square \eqref{E:MV-square} is a tor-independent
    Mayer--Vietoris square and $f$ is $f$-flat (e.g., 
    flat), then the square above is a tor-independent Mayer--Vietoris square. This claim 
    follows from the same argument.
  \end{example}
  In the next Proposition, we show that general Mayer--Vietoris squares can smooth-locally 
  be dominated by much simpler ones. 
  \begin{proposition}\label{P:qaff_dom}
    Fix a weak Mayer--Vietoris square as in \eqref{E:MV-square}. Then 
    smooth-locally on $X$, there is an \'etale neighborhood $p\colon X''\to X'$
    of $Z'$
    such that the composition $f\circ p\colon X'' \to X$ is quasi-affine. 
  \end{proposition}
  \begin{proof}
    By Lemma \ref{L:mv_bc}, we may assume that $X$ is an affine scheme.
    Observe that 
    the Deligne--Mumford locus of $X'$  is an open substack containing $Z'$. 
    In particular, there exist an affine scheme $V$ and an \'etale morphism $V \to X'$ whose image contains $Z'$. Let $Z_V = V\times_{X'} 
  Z'$; then the composition $z\colon Z_V \to Z' \simeq Z=X\smallsetminus U$ is affine and \'etale. After passing to an \'etale 
  cover of $X$, we may assume that the morphism $Z_V \to Z' \simeq Z$ has a section 
  $s\colon Z \to Z_V$. Since $z$ is \'etale and separated, $s$ is an open and closed 
  immersion; it follows that $X''=V \smallsetminus (Z_V \smallsetminus s(Z))$ is an open 
  subscheme of $V$. After replacing $X''$ with a quasi-compact open neighborhood
  of $s(Z)$, we can assume that $X''$ is quasi-compact.
  Thus, $X'' \to X' \to X$ is a quasi-affine morphism and $X'' \to X'$ is an isomorphism over $Z'$.
  \end{proof}
  The following is the last lemma of the section.
  \begin{lemma}\label{L:supp_flat}
    Fix a tor-independent Mayer--Vietoris square as in \eqref{E:MV-square}. If $j$ is 
    quasi-compact, then $\QCOH_Z(X) \subseteq \QCOH_{f-\mathrm{fl}}(X)$. 
  \end{lemma}
  \begin{proof}
    We may assume that $X$ is affine and that $i\colon Z\to X$ is finitely presented. In this case if $N\in \QCOH_Z(X)$, then we may 
    write $N$ as the union of its quasi-coherent subsheaves of finite type and these also belong to 
    $\QCOH_Z(X)$. Thus, it is sufficient to prove the result for such sheaves. By Lemma 
    \ref{L:fp_thickenings}\itemref{LI:fp_thickenings:ft}, there is an $n\gg 0$ such that $N 
      \to i_*^{[n]}(i^{[n]})^*N$ is an isomorphism. The result now follows from Lemma 
      \ref{L:flatness_conc_cl}\itemref{LI:flatness_conc_cl:conc}.
  \end{proof}

  Flat Mayer--Vietoris squares and
  weak Mayer--Vietoris squares are stable under arbitrary base change
  but tor-independent Mayer--Vietoris squares are not. We now give six examples
  of squares as in \eqref{E:MV-square}:
  \begin{itemize}
    \item Two examples of weak Mayer--Vietoris squares that are not
      tor-independent Mayer--Vietoris squares. (Example~\ref{E:weak-but-not-MV})
    \item Two examples of tor-independent Mayer--Vietoris squares
      where $f$ is a non-flat closed immersion. (Examples~\ref{E:diagonal-square-not-flat}--\ref{EX:valuation-MV-square})
    \item A tor-independent Mayer--Vietoris square that is not universally a tor-independent Mayer--Vietoris square. (Example~\ref{E:mv-not-umv})
    \item A flat Mayer--Vietoris square, with $j$ not quasi-compact, that 
      is not a pushout in the category of affine schemes. (Example~\ref{E:fmv-nonqc-nonpo})
  \end{itemize}
  As we will see later, tor-independent Mayer--Vietoris squares satisfy gluing of
  quasi-coherent sheaves. In particular,
  $\Gamma(X,\Orb_X)=\Gamma(X',\Orb_{X'})\times_{\Gamma(U',\Orb_{U'})}\Gamma(U,\Orb_U)$,
  which does not always hold for weak Mayer--Vietoris squares.

  \begin{example}\label{E:weak-but-not-MV}
    Let $A=k[x]$, $B=A[z_1,z_2,\dots]/(xz_1,\{z_k-xz_{k+1}\}_{k\geq 1})$
    and $C=B/(z_1)$. Then 
    $A/(x^n) = B/(x^n) = C/(x^n) = k[x]/(x^n)$ and $A_x = B_x = C_x = k[x]_x$.
    Let $X=\spec A$, $Z=\spec A/(x)$, $U=X\smallsetminus Z$,
    $X'=\spec B$, $U'=X'\smallsetminus Z$ and $X''=\spec C$. Then the squares
    \[
    \vcenter{\xymatrix{
        U'\ar[r]\ar[d] & X'\ar[d] & & U''\ar[d] \ar[r] & X'' \ar[d] \\
        U\ar[r] & X\ar@{}[ul]|\square & & U' \ar[r] & X' \ar@{}[ul]|\square
      }}
    \]
    are weak Mayer--Vietoris squares but not
    tor-independent Mayer--Vietoris squares.
    Indeed $A\to B\times_{B_x} A_x=B$ and
    $B\to C\times_{C_x} B_x=C$ are not isomorphisms.

    Note that $Z\inj X$ is a Cartier divisor but $Z\inj X'$ is not a Cartier
    divisor.
  \end{example}

  \begin{example}\label{E:diagonal-square-not-flat}
    A diagonal Mayer--Vietoris square (Example~\ref{E:diagonal-square-is-MV})
    is typically not flat, e.g., let $A$ be a noetherian
    ring, $I\subseteq A$ an ideal and consider the $I$-adic completion
    $\widehat{A}_I$. Let $f\colon X'=\spec \widehat{A}_I\to X=\spec A$ and
    $j\colon U=X\smallsetminus V(I)\to X$. This gives rise to a flat Mayer--Vietoris square as in 
    \eqref{E:MV-square} and the diagonal Mayer--Vietoris square is tor-independent. 
    However, the closed immersion $\Delta_f$ is usually not flat except in
    cases such as when $A$ is already $I$-adically complete. For example,
    $\Delta_f$ is not flat when $A$ is an integral domain, finitely generated
    over a field $k$, and $I\neq 0$, $I\neq A$.
  \end{example}
    
  \begin{example}\label{EX:valuation-MV-square}
    Let $V$ be a valuation ring with valuation $\nu\colon K(V)^\times\to
    \Gamma$ and let $x\in V$ be non-zero. Then $V_x$ is also a valuation ring
    and $V_x=V_P$
    where $P\subseteq V$ is the maximal prime ideal properly contained in the
    prime ideal $Q=\sqrt{(x)}$. Explicitly:
    \begin{align*}
    P &= \{a\in V\;: \forall n\in \N\;:\; \nu(a) > n\nu(x)\}\\
    Q &= \{a\in V\;: \exists n\in \N \;:\; n\nu(a) \geq \nu(x)\}.
    \end{align*}
    Let $X=\spec(V)$, $U=\spec(V_P)$, $Z=\spec(V/xV)$ and $X'=\spec(V/P)$.
    The resulting square as in \eqref{E:MV-square} is tor-independent. Indeed,
    it is a weak Mayer--Vietoris square since $P\subseteq (x^n)$ for all $n$.
    It remains to verify that $\Tor_i^A(V/xV,V/P)=0$ for all $i>0$.
    But $x\notin P$ so $x$ is $V/P$-regular, hence the Tors vanish.
  \end{example}

  \begin{example}\label{E:mv-not-umv}
    Consider the valuation $\nu\colon k(x,y)^\times\to \Z^2$ with $\nu(x)=(0,1)$
    and $\nu(y)=(1,0)$ where $\Z^2$ is lexicographically ordered. The
    corresponding valuation ring $V$ has three prime ideals: the maximal
    ideal $Q=(x)$, the prime ideal $P=(y,y/x,y/x^2,\dots)$ and the zero ideal.
    Then $X'=\spec V/P\to X=\spec V$, $U=\spec V_P=\spec V_x$ is
    a tor-independent Mayer--Vietoris square as in the previous example.

    Let $A=V/yV$ and let $z_n=y/x^n$ denote the image of $y/x^n$ in $A$. Then
    $A=k[x,z_1,...]/(xz_1,z_k-xz_{k+1})_{(x,z_1,z_2,...)}$ and
    $B=A/PA=k[x]_{(x)}$. Let $Y'=\spec B$, $Y=\spec A$ and $U=\spec A_x$.
    As in Example~\ref{E:weak-but-not-MV}, $A/(x^n)=k[x]/(x^n)=B/(x^n)$
    but $A\to B\times_{B_x} A_x=B$ is not an isomorphism.
  \end{example}

  \begin{example}\label{E:fmv-nonqc-nonpo}
    Let $X=\spec A$ be the spectrum of an absolutely flat ring such that there exists
    a non-discrete point $x\in |X|$. Let $\mathfrak{m}\subseteq A$ be the corresponding
    maximal ideal. For a concrete example, let $\mathbb{P}$ be the set
    of primes of $\Z$, let $A=\prod_{p\in \mathbb{P}} \mathbb{F}_p$ and
    let $\mathfrak{m}$ be a maximal ideal containing the ideal $\oplus_{p\in \mathbb{P}} 
    \mathbb{F}_p$.
    % http://math.stackexchange.com/users/4396/wxu), Infinite product of fields, URL 
    % (version: 2015-09-29): http://math.stackexchange.com/q/88062
    %
    Let $X' = \spec A/\mathfrak{m}$ and let $f\colon X' \to X$ be the induced
    closed immersion. Let $j\colon U=X\smallsetminus \{x\} \to X$ be the open immersion of
    the complement. Since $A$ is an absolutely flat ring, $f$ is also flat. Let 
    $U'=X'\times_X U = \emptyset$; then the resulting square is a flat Mayer--Vietoris 
    square but $j$ is not quasi-compact.

    Note that the natural map $|X'|\amalg |U|=|X'|\amalg_{|U'|} |U|\to |X|$ is
    not a homeomorphism since $|X'|\subset |X|$ is not open. In particular, the
    functor $\Phi_{\Et}$ is not an equivalence, cf.\ Corollary~\ref{C:wMV-is-univ-submersive}.

    Let $B = \Gamma(U,\Orb_U)$. If the square was a pushout in the category of 
    affine schemes, then corresponding to the  maps $X'\to X'$ and $U \to 
    \spec B$, there would be a unique map
    $g\colon X \to \spec B\amalg X'=\spec (B\times A/\mathfrak{m})$. Then
    $g^{-1}(X')=X'$ which is a contradiction since $X'\subseteq X$ is not open.
    
    This 
    example also shows that 
    \[
    \Gamma(X,\Orb_X) \to \Gamma(U,\Orb_U) \times \Gamma(X',\Orb_{X'})
    \]
    is not an isomorphism. In particular, the functor $\Phi_{\QCOH}$ is not even fully faithful.
  \end{example}
\section{Gluing of modules in Mayer--Vietoris squares}\label{S:QCoh-gluing}
In this section, we show that quasi-coherent sheaves of modules, and related
objects such as quasi-coherent sheaves of algebras, can be glued in tor-independent
Mayer--Vietoris squares. This generalizes previous results of Ferrand--Raynaud~\cite[App.]{MR0272779} and Moret-Bailly \cite{MR1432058}. We will prove this using some ideas from the theory
developed in \cite[\S5]{perfect_complexes_stacks} for triangulated categories that are perfectly suited to simultaneously deal with the non-flatness of $f$ and the non-affineness of $j$.  For quasi-compact and quasi-separated algebraic spaces and
in the context of stable $\infty$-categories, this was recently
accomplished (independently) by Bhatt \cite[Prop.~5.6]{2014arXiv1404.7483B}. Since we work with morphisms of algebraic stacks that may not have finite cohomological dimension, we do not expect gluing results to hold in this generality in the unbounded derived category. Before we get to gluing, we characterize the tor-independent squares in terms of derived categories.
\begin{notation}
  Let $i\colon Z \inj X$ be a closed immersion of algebraic stacks with complement $j\colon 
  U \to X$. Define
  \[
  \DQCOH[,Z](X) = \{ M \in \DQCOH(X) \suchthat \LDERF \QCPBK{j}M \simeq 0\}.
  \]
\end{notation}
Recall that $f\colon X'\to X$ is \emph{concentrated} if $f$ is quasi-compact,
quasi-separated and has universal finite cohomological
dimension~\cite[Def.~2.4]{perfect_complexes_stacks}.
\begin{proposition}\label{P:MV_SQ_MB}
  Fix a cartesian diagram as in \eqref{E:MV-square} with $f$ concentrated, $j$ 
  quasi-compact and $X$ quasi-compact and quasi-separated. Consider the
  following conditions:
  \begin{enumerate}
    \item \label{PI:MV_SQ_MB:timv} the square is a tor-independent Mayer--Vietoris square;
    \item \label{PI:MV_SQ_MB:tdqcb} $\RDERF f_*$ and $\LDERF
      \QCPBK{f}$ induce $t$-exact equivalences
      $\DQCOH[,Z]^b(X)\simeq \DQCOH[,Z']^b(X')$;
    \item \label{PI:MV_SQ_MB:tdqc} $\RDERF f_*$ and $\LDERF
      \QCPBK{f}$ induce $t$-exact equivalences
      $\DQCOH[,Z](X)\simeq \DQCOH[,Z'](X')$; and
    \item \label{PI:MV_SQ_MB:dqc} $\RDERF f_*$ and $\LDERF
      \QCPBK{f}$ induce equivalences 
      $\DQCOH[,Z](X)\simeq\DQCOH[,Z'](X')$.
  \end{enumerate}
  Then
  \itemref{PI:MV_SQ_MB:timv}$\implies$\itemref{PI:MV_SQ_MB:tdqcb}$\iff$\itemref{PI:MV_SQ_MB:tdqc}$\implies$\itemref{PI:MV_SQ_MB:dqc}.
  Moreover, if $f_Z$ is affine, then all conditions are equivalent; and if $f_Z$
  is representable and $Z$ has quasi-affine diagonal, then
  \itemref{PI:MV_SQ_MB:tdqcb}$\implies$\itemref{PI:MV_SQ_MB:timv}.
\end{proposition}
In the application of Proposition \ref{P:MV_SQ_MB} to the main result of this section (Theorem \ref{T:MV-QCoh-gluing}), we will only need \itemref{PI:MV_SQ_MB:timv}$\implies$\itemref{PI:MV_SQ_MB:dqc} when $X$ is an affine scheme and $X'$ is a quasi-affine scheme. We have included the general situation for independent interest. Note that condition \itemref{PI:MV_SQ_MB:dqc} is the definition of a Mayer--Vietoris
$\DQCOH$-square~\cite[Def.~5.5]{perfect_complexes_stacks} and Proposition \ref{P:MV_SQ_MB}\itemref{PI:MV_SQ_MB:timv}$\Rightarrow$\itemref{PI:MV_SQ_MB:dqc} gives another proof
of~\cite[Ex.~5.6]{perfect_complexes_stacks}. 
\begin{remark}
  Taking $f\colon \spec k\to BG$ as in
  \cite[Rem.~1.6]{hallj_dary_alg_groups_classifying} and $U=\emptyset$
  provides an example where \itemref{PI:MV_SQ_MB:tdqcb} is satisfied,
  but \itemref{PI:MV_SQ_MB:timv} is not satisfied ($f$ is representable but
  not affine
  and its target does not have affine stabilizers).
\end{remark}
\begin{proof}[Proof of Proposition~\ref{P:MV_SQ_MB}]
  Trivially, \itemref{PI:MV_SQ_MB:tdqc} implies \itemref{PI:MV_SQ_MB:tdqcb} and \itemref{PI:MV_SQ_MB:dqc}. Since $X$ is quasi-compact and quasi-separated, we may assume that there is a finitely presented complement $i\colon Z \inj X$ of $U$~\cite[Prop.~8.2]{rydh-2014}.  If \itemref{PI:MV_SQ_MB:tdqcb} is satisfied, then
  $\LDERF \QCPBK{f}i_*\Orb_Z=f^*i_*\Orb_Z$ and the adjunction maps
  $i_*\Orb_Z\to f_*f^*i_*\Orb_Z$ and $f^*f_*i'_*\Orb_{Z'}\to i'_*\Orb_{Z'}$ are
  isomorphisms. If $f_Z$ is affine, then \itemref{PI:MV_SQ_MB:timv} holds (Lemma \ref{L:MV-open-qc}\itemref{LI:MV-open-qc:ti}). Otherwise, if $Z$
  has quasi-affine diagonal, then we start by noting that the
  $t$-exactness of $\RDERF
  f_*$ also shows that $\RDERF (f_Z)_*$ is $t$-exact.
  By~\cite[Lem.~2.2(vi)]{perfect_complexes_stacks}, it follows
  that if $\tilde{Z} \to Z$ is a smooth morphism, where $\tilde{Z}$ is
  an affine scheme, then the pullback $f_{\tilde{Z}}$ of $f_Z$ to
  $\tilde{Z}$ is such that $\RDERF (f_{\tilde{Z}})_*$ is
  $t$-exact. Since $f_{\tilde{Z}}$ is representable, we conclude that
  $f_{\tilde{Z}}$ is affine from Serre's Criterion \cite[Thm.~8.7]{rydh-2009}.
  By smooth descent, $f_Z$ is affine, and we
  again see that \itemref{PI:MV_SQ_MB:timv} holds. For
  \itemref{PI:MV_SQ_MB:dqc}$\implies$\itemref{PI:MV_SQ_MB:tdqcb} when $f_Z$ is
  affine, it is sufficient to prove that $\RDERF^m f_*N = 0$ for all $m>0$ and
  $N\in \QCOH_{Z'}(X')$. Since $\RDERF^mf_*(-)$ is compatible with filtered
  colimits, by writing $N$ as a union of its finite type subsheaves
  \cite{rydh-2014}, we are reduced to proving the assertion when $N$ is of
  finite type. In this case, there is an $n>0$ such that $N \simeq
  i'^{[n]}_*(i'^{[n]})^*N$ (Lemma
  \ref{L:fp_thickenings}\itemref{LI:fp_thickenings:ft}). Then $\RDERF^m f_*N
  \simeq \RDERF^m f_*i'^{[n]}_*(i'^{[n]})^*N \simeq \RDERF^m
  (f_{Z^{[n]}})_*(i'^{[n]})^*N = 0$ for all $m>0$ as $f_Z$, and so
  $f_{Z^{[n]}}$, is affine.

  We will finish the proof by showing that \itemref{PI:MV_SQ_MB:timv}$\implies$\itemref{PI:MV_SQ_MB:tdqcb}$\implies$\itemref{PI:MV_SQ_MB:tdqc}.
  For every ${M} \in \DQCOH[,Z](X)$
  and ${N} \in \DQCOH[,Z'](X')$ we have adjunction maps
  \[
  \eta_{{M}} \colon {M} \to \RDERF f_*\LDERF
  \QCPBK{f} {M} \quad\mbox{and}\quad \epsilon_{{N}} \colon \LDERF
  \QCPBK{f} \RDERF f_* {N} \to {N}.
  \]
  We will show that these are quasi-isomorphism. 

  For \itemref{PI:MV_SQ_MB:timv}$\implies$\itemref{PI:MV_SQ_MB:tdqcb} it is enough---by standard truncation
  arguments---to prove:
  \begin{itemize}
    \item $\eta_{{M}[0]}$ and $\epsilon_{{N}[0]}$ are
  quasi-isomorphisms; and
  \item $\LDERF f^*({M}[0])\to (f^*{M})[0]$
    and $(f_*{N})[0]\to \RDERF f_*({N}[0])$ are
    quasi-isomorphisms,
  \end{itemize}
  where ${M}$ is a quasi-coherent $\Orb_X$-module such
  that $j^*{M} \cong 0$ and ${N}$ is a quasi-coherent $\Orb_{X'}$ such
  that $j'^*{N} \cong 0$. If $M$ and $N$ are of finite type, then there exists an integer 
  $n\gg 0$ such
  that $M \to i^{[n]}_*(i^{[n]})^*M$ and $N \to i'^{[n]}_*(i'^{[n]})^*N$ are isomorphisms 
  (Lemma \ref{L:fp_thickenings}\itemref{LI:fp_thickenings:ft}). Now Lemma 
  \ref{L:flatness_conc_cl} informs us that ${f_{Z^{[n]}} \colon Z'^{[n]} \to Z^{[n]}}$ is an 
  isomorphism, $i^{[n]}$ and $f$ are tor-independent, and $M$ is $f$-flat. This immediately proves the claims 
  when $M$ and $N$ are of finite type. But every quasi-coherent sheaf on $X$ or 
  $X'$ is a directed limit of its quasi-coherent subsheaves of finite type \cite{rydh-2014}, so we have the 
  claim in general.

  To see
  that \itemref{PI:MV_SQ_MB:tdqcb}$\implies$\itemref{PI:MV_SQ_MB:tdqc} it is enough to prove that $\LDERF \QCPBK{f}$ is left
  $t$-exact on $\DQCOH[,Z](X)$ and that $\RDERF f_*$ is right
  $t$-exact on $\DQCOH[,Z'](X')$. For the first claim, let
  ${M}$ be a complex in $\DQCOH[,Z]^{\geq 0}(X)$. We may write
  ${M}$ as a homotopy colimit of its truncations $\trunc{\leq
    n}{M}$. Since $\LDERF \QCPBK{f}$ commutes with coproducts and is
  $t$-exact on $\DQCOH[,Z]^b(X)$, it follows that $\LDERF \QCPBK{f}
  {M}\in \DQCOH[,Z']^{\geq 0}(X')$ so $\LDERF \QCPBK{f}$ is
  $t$-exact. Also, if ${N}$ is a complex in $\DQCOH[,Z']^{\leq
    0}(X')$, then since $f$ is concentrated and $X$ is quasi-compact and quasi-separated, there exists an integer $n$ such that $\trunc{>0}\RDERF f_*{N} \to \trunc{>0}\RDERF
  f_*(\trunc{\geq -n}{N}) \homotopic 0$ is a quasi-isomorphism. Hence, $\RDERF f_*$ is $t$-exact. 
\end{proof}
The following theorem generalizes~\cite[Thm.~3.1]{MR1432058} ($f$ affine)
and \cite[App.]{MR0272779} ($f$ affine and flat) and is Theorem 
\ref{MT:glue_timv}\itemref{MTI:glue_timv:qcoh}. 
\begin{theorem}\label{T:MV-QCoh-gluing}
Fix a tor-independent Mayer--Vietoris square as in \eqref{E:MV-square} with $j$ quasi-compact. 
The functors 
\[
\Phi_{\MOD} \colon \MOD(X) \rightleftarrows \MOD(X') \times_{\MOD(U')} \MOD(U) \colon \Psi
\]
where
\[
\Phi_{\MOD}(M) = (f^*M,j^*M,\delta) \quad \mbox{and} \quad \Psi(M',M_U,\alpha) = f_*M'\times_\alpha j_*M_U
\]
and $\delta$ is the canonical isomorphism $j'^*f^*M\cong f_U^*j^*M$, are adjoint. Also, $\Phi_{\MOD}$ preserves tensor products and the restriction of $\Phi_{\MOD}$ to $\QCOH_{f-\mathrm{fl}}(X)$ induces an equivalence of categories
\begin{equation}
\Phi_{\QCOH,f-\mathrm{fl}}\colon \QCOH_{f-\mathrm{fl}}(X)\to \QCOH(X')\times_{\QCOH(U')} \QCOH_{f_U-\mathrm{fl}}(U)\label{TE:MV-QCoh-gluing}
\end{equation}
that preserves short exact sequences. Moreover,
\begin{enumerate}
\item \label{TI:MV-QCoh-gluing:supp}$f^* \colon \QCOH_Z(X) \to \QCOH_{Z'}(X')$ is an equivalence;
\item $\Phi_{\QCOH}$ preserves and reflects
  \begin{enumerate}
  \item \label{TI:MV-QCoh-gluing:zero} zero objects,
  \item \label{TI:MV-QCoh-gluing:surj} surjective homomorphisms and 
  \item \label{TI:MV-QCoh-gluing:ft} modules of finite type; and
  \end{enumerate}
\item $\Phi_{\QCOH,f-\mathrm{fl}}$ preserves and reflects
  \begin{enumerate}
  \item \label{TI:MV-QCoh-gluing:fp} modules of finite presentation and
  \item \label{TI:MV-QCoh-gluing:flat} flat modules.
  \end{enumerate}
\end{enumerate}
\end{theorem}
\begin{proof}
  That $\Phi_{\MOD}$ and $\Psi$ are adjoints is clear.
  Hence, to prove that \eqref{TE:MV-QCoh-gluing} is an equivalence, it is 
  enough to show that the unit $M\to \Psi(\Phi_{\MOD}(M))$ and the counit 
  $\Phi_{\MOD}(\Psi(M',M_U,\delta))\to
  (M',M_U,\delta)$ of the adjunction are isomorphisms when restricted to the relevant 
  subcategories. This is smooth local on $X$, so we may assume that $X$ is an affine 
  scheme.

  Until further notice, we will assume that $X'$ is a quasi-compact and quasi-separated 
  algebraic space (even quasi-affine scheme is sufficient). Now the functor
  \[
  \LDERF f^* \colon \DQCOH[,Z](X) \to \DQCOH[,Z'](X')
  \]
  is a $t$-exact equivalence of categories (Proposition \ref{P:MV_SQ_MB}). Thus, we have a 
  Mayer--Vietoris $\DQCOH$-square in the sense of 
  \cite[Def.~5.5]{perfect_complexes_stacks}, which provides some natural 
  distinguished triangles~\cite[Lem.~5.9]{perfect_complexes_stacks} that we now describe.
  \begin{enumerate}
  \item[(i)] For every $M\in\DQCOH(X)$, there is a distinguished triangle:
    \[ 
    \xymatrix@C2pc{%
      M \ar[r]
      & \RDERF j_* \LDERF j^* M \oplus \RDERF f_*\LDERF f^*M \ar[r]
      & \RDERF f_*\LDERF f^*\RDERF j_*\LDERF j^*M\ar[r]
      & M[1]. }
    \]
  \item[(ii)] Conversely, given $M_U \in \DQCOH(U)$, $M' \in
    \DQCOH(X')$ and an isomorphism $\delta \colon \LDERF j'^*M' \to
    \LDERF f_U^*M_U$, we define $M$ by the following distinguished triangle in
    $\DQCOH(X)$: 
    \[
    \xymatrix{%
      M \ar[r]
       & \RDERF j_*M_U \oplus \RDERF f_*M' \ar[rr]^{\begin{psmallmatrix}\eta^f_{\RDERF j_*M_U} & -\alpha\end{psmallmatrix}}
       && \RDERF f_*\LDERF f^*\RDERF j_*M_U\ar[r] & M[1],}
    \]
    where $\alpha \colon \RDERF f_*M' \to \RDERF f_*\LDERF f^*\RDERF j_*M_U$ is the composition:
    \begin{align*}
    \RDERF f_*M' \xrightarrow{\RDERF f_*\eta^{j'}_{M'}}& \RDERF f_*\RDERF j'_*\LDERF  j'^*M'\\
    \xrightarrow{\RDERF f_*\RDERF j'_*\delta}& \RDERF f_*\RDERF j'_*\LDERF f_U^*M_U \cong \RDERF f_*\LDERF f^*\RDERF j_*M_U.
    \end{align*}    
    Then the induced maps $\LDERF j^*M
    \to M_U$ and $\LDERF f^*M \to M'$ are isomorphisms. 
  \end{enumerate}
  Now let $M \in \QCOH_{f-\mathrm{fl}}(X)$; then the distinguished triangle from (i) reduces 
  to the following distinguished triangle:
  \[
  \xymatrix@C2pc{%
      M \ar[r]
      & \RDERF j_* j^*M \oplus \RDERF f_*f^*M \ar[r]
      & \RDERF f_*\LDERF f^*\RDERF j_*j^*M\ar[r]
      & M[1]. }
  \]
  Observe that tor-independent base change \cite[Cor.~4.13]{perfect_complexes_stacks} 
  implies that:
  \[
  \RDERF f_*\LDERF f^*\RDERF j_*j^*M \simeq \RDERF f_*\RDERF j'_* \LDERF f^*_U \LDERF j^*M \simeq \RDERF k_*\LDERF j'^*\LDERF f^*M \simeq \RDERF k_*k^*M
  \]
  where $k=j\circ f_U$.
  Hence, taking the long exact cohomology sequence of the distinguished triangle above, we obtain the following exact sequence:
  \[
  0 \to M \to j_*j^*M \oplus f_*f^*M \to k_*k^*M \to 0.
  \]
  So the natural map $M \to \Psi(\Phi_{\MOD}(M))$ is an 
  isomorphism when $M$ is quasi-coherent and $f$-flat.

  Conversely, given a triple $(M',M_U,\delta)$, where $M' \in \QCOH(X')$ and $M_U \in \QCOH_{f_U-\mathrm{fl}}(U)$, (ii) provides a distinguished
  triangle:
  \[
  \xymatrix{%
    M \ar[r] & \RDERF j_* M_U \oplus \RDERF  f_* M' \ar[r] & \RDERF k_*M'_U\ar[r] & M[1],}
  \]
  such that the induced maps $\LDERF j^*M \to M_U$ and $\LDERF f^*M \to M'$
  are isomorphisms. Since $\Phi(M',M_U,\delta)=\COHO{0}(M)$, it is enough to
  show that $M$ is concentrated in degree $0$. To see this, we have a distinguished triangle:
  \[
  \xymatrix{%
    \COHO{0}(M)[0] \ar[r] & M \ar[r]
    & \trunc{\geq 1}(M) \ar[r] & \COHO{0}(M)[1].}
  \]
  If we apply the $t$-exact functor $\LDERF j^*$ to this triangle, then the third term
  vanishes so $\trunc{\geq 1}(M)\in \DQCOH[,Z](X)$. If we instead apply the
  right $t$-exact functor $\LDERF f^*$ to this triangle, we obtain the
  triangle:
  \[
  \xymatrix{%
    \LDERF f^*\COHO{0}(M)[0] \ar[r] & M'[0] \ar[r]
    & \LDERF f^*\trunc{\geq 1}(M) \ar[r] & \LDERF f^*\COHO{0}(M)[1].}
  \]
  The first two terms are concentrated in degrees $\leq 0$ and the third is
  concentrated in degrees $\geq 1$ since $\DQCOH[,Z](X)\to
  \DQCOH[,Z'](X')$ is a $t$-exact equivalence. It follows that
  $\trunc{\geq 1}(M)\simeq 0$. 

  Hence, we have proven the equivalence \eqref{TE:MV-QCoh-gluing} when $f\colon X' \to X$ is 
  quasi-compact, quasi-separated and representable. We now address the general case. By 
  Proposition \ref{P:qaff_dom}, smooth-locally on $X$ there is an \'etale neighborhood 
  $X''$ of $Z'$ in $X'$ such that the induced composition $w\colon X'' \to X$ is quasi-affine. Let $U''=X''\times_X U$.
  It now follows 
  from the case considered already, as well as \cite[Ex.~1.2]{MR2774654}, that we have 
  equivalences:
  \begin{align*}
  \QCOH_{w-\mathrm{fl}}(X) &\simeq \QCOH(X'') \times_{\QCOH(U'')} \QCOH_{w_U-\mathrm{fl}}(U)\\
    &\simeq \left(\QCOH(X'') \times_{\QCOH(U'')} \QCOH(U') \right)
      \times_{\QCOH(U')}  \QCOH_{w_U-\mathrm{fl}}(U)\\
    &\simeq \QCOH(X')\times_{\QCOH(U')} \QCOH_{w_U-\mathrm{fl}}(U). 
  \end{align*}
  Note that $\QCOH_{f-\mathrm{fl}}(X)\subseteq \QCOH_{w-\mathrm{fl}}(X)$ is
  a full subcategory and that we have an equivalence 
  \[
  \QCOH_{f-\mathrm{fl}}(X)\to \QCOH_{w-\mathrm{fl}}(X)\times_{\QCOH_{w_U-\mathrm{fl}}(U)} \QCOH_{f_U-\mathrm{fl}}(U).
  \]
  It follows that $\Phi_{\QCOH,f-\mathrm{fl}}$ is an equivalence 
  and it preserves short exact sequences.
 
  Now for \itemref{TI:MV-QCoh-gluing:supp}: if $M \in \QCOH_Z(X)$, then $M$ is $f$-flat (Lemma \ref{L:supp_flat}). Hence,
  \[
  M \to \Psi(f^*M,0,0)=f_*f^*M
  \]
  is an isomorphism. Also if $M' \in \QCOH_{Z'}(X')$, then
  \[
  (f^*f_*M',0,0)=\Phi_{\MOD}(f_*M') \to (M',0,0)
  \]
  is an isomorphism and the claim follows.
  
  For \itemref{TI:MV-QCoh-gluing:zero}: the preservation is obvious. For the reflection: if 
  $M \in \QCOH(X)$ and $j^*M\cong 0$, then $M \in \QCOH_Z(X)$. But if $f^*M \cong 0$ 
  too, then $M \cong 0$ by \itemref{TI:MV-QCoh-gluing:supp}.

  For \itemref{TI:MV-QCoh-gluing:surj}: the preservation is because $\Phi_{\MOD}$ 
  admits a right adjoint $\Psi$ and so is right exact. For the reflection: if $u\colon M \to N$ 
  is a morphism in $\QCOH(X)$ and $j^*u$ and $f^*u$ are surjective, then $j^*\coker(u) = 
  0$ and $f^*\coker(u) = \coker(f^*u) = 0$. It follows from 
  \itemref{TI:MV-QCoh-gluing:zero} that $\coker(u) = 0$ and $u$ is surjective.
  
   For \itemref{TI:MV-QCoh-gluing:ft}: the preservation is clear. For the reflection: we may 
   assume that $X$ is affine. Write $M\in \QCOH(X)$ as a filtered union of quasi-coherent subsheaves 
   $M_\lambda$ of finite type. For sufficiently large $\lambda$ we see that 
   $\Phi_{\MOD}(M_\lambda) \to \Phi(M)$ is surjective. By \itemref{TI:MV-QCoh-gluing:surj}, 
   we see that $M_\lambda = M$ and so $M$ is of finite type.

   For \itemref{TI:MV-QCoh-gluing:fp}: the preservation is clear. For the reflection:   Let $M$ be an $f$-flat quasi-coherent $\Orb_X$-module such that $\Phi(M)$ is
  of finite presentation. By  \itemref{TI:MV-QCoh-gluing:ft} we know that $M$ is of finite 
  type. Since we are free to assume that $X$ is affine, there is an exact sequence $0\to 
  K\to \Orb_X^{\oplus n}\to M\to 0$. But $M$ is $f$-flat, so the sequence remains exact after 
  applying $f^*$. Since $\Phi(K)$ is of finite type, so is $K$ and hence $M$ is of finite
  presentation.

  For \itemref{TI:MV-QCoh-gluing:flat}: the preservation is clear. For the reflection: as before, we 
  may assume that $X$ is affine and $f$ is quasi-affine. Let $M \in \QCOH_{f-\mathrm{fl}}(X)$ be such that $\Phi_{\QCOH,f-\mathrm{fl}}(M)$ is flat. Let $N \in \QCOH(X)$. It is sufficient to prove that
  $\trunc{<0}(N\tensor^{\LDERF}_{\Orb_X} M) \simeq 0$. We begin with the following 
  distinguished triangle:
  \[
\xymatrix{  C  \ar[r] & N \ar[r] & \RDERF j_*j^*N \ar[r] & C[1].}
  \]
  Observe that the derived projection formula \cite[Cor.~4.12]{perfect_complexes_stacks} 
  implies that
  \[
  (\RDERF j_*j^*N) \tensor^{\LDERF}_{\Orb_X} M \simeq \RDERF j_*((j^*N)\tensor^{\LDERF}_{\Orb_{U}} j^*M).
  \]
  But $j^*M$ is flat and so we conclude immediately that $\trunc{<0}((\RDERF j_*j^*N) 
  \tensor^{\LDERF}_{\Orb_X} M) \simeq 0$. It remains to prove that $\trunc{<0}(C\tensor^{\LDERF}_{\Orb_X} M) \simeq 0$. To this end, we first note that $\trunc{<0}C \simeq 0$ and $j^*C \simeq 0$. Moreover, 
  \[
  \LDERF f^*(C\tensor^{\LDERF}_{\Orb_X} M) \simeq (\LDERF f^*C) \tensor^{\LDERF}_{\Orb_{X'}} \LDERF f^*M \simeq (\LDERF f^*C) \tensor^{\LDERF}_{\Orb_{X'}} f^*M.
  \]
  By assumption, $f^*M$ is flat and so for all integers $k$ there are isomorphisms:
  \[
  \COHO{k}\bigl((\LDERF f^*C) \tensor^{\LDERF}_{\Orb_{X'}} f^*M\bigr) \cong \COHO{k}(\LDERF f^*C) \tensor_{\Orb_{X'}} f^*M.
  \]
  But $C \in \DQCOH[,Z](X)$, so $\trunc{<0}C \simeq 0$ implies $\trunc{<0}(\LDERF f^*C) \simeq 0$ (Proposition \ref{P:MV_SQ_MB}). Putting this all together, we see that $\trunc{<0}(\LDERF f^*(C \tensor_{\Orb_X}^{\LDERF} M)) \simeq 0$. But $j^*(C\tensor^{\LDERF}_{\Orb_X} M) \simeq 0$, which implies that $\trunc{<0}(C\tensor_{\Orb_X}^{\LDERF} M) \simeq 0$ (Proposition \ref{P:MV_SQ_MB} again). 
\end{proof}
Note that \ref{T:MV-QCoh-gluing}\itemref{TI:MV-QCoh-gluing:flat} gives a vast generalization of \cite[Prop.~4.1(iii)]{MR1432058}, where only the descent of \'etaleness is proved. 
\begin{remark}
  Assume that we are in the situation of Theorem \ref{T:MV-QCoh-gluing}. If $f$ is 
  concentrated, then the Mayer--Vietoris triangle shows
  that the functor $\DQCOH(X)\to \DQCOH(X')\times_{\DQCOH(U')} \DQCOH(U)$ is
  essentially surjective. It is, however, not fully faithful. The reason is
  a well-known fault of the derived category: whereas cones are unique up
  to isomorphism, morphisms between cones are not unique. One way to fix this
  problem is to work with $\infty$-categories. Then one obtains the expected
  equivalence, cf.~\cite[Prop.~5.6]{2014arXiv1404.7483B}.
\end{remark}

We now have a number of corollaries.
\begin{corollary}\label{C:MV-QCoh-gluing:ff}
  Assume that we are in the situation of Theorem \ref{T:MV-QCoh-gluing}.   If
  ${M\in \MOD(X)}$ and $N\in \QCOH_{f-\mathrm{fl}}(X)$, then the natural
  map:
  \[
  \Hom(M,N)\to \Hom(f^*M,f^*N)\times_{\Hom(j'^*f^*M,j'^*f^*N)}
  \Hom(j^*M,j^*N)
  \]
  is bijective.
\end{corollary}
\begin{proof}
Follows from the unit $N\to \Psi(\Phi_{\MOD}(N))$ being an isomorphism.
\end{proof}
\begin{corollary}\label{C:MV-QCoh-gluing:flatness}
  Assume that we are in the situation of Theorem \ref{T:MV-QCoh-gluing}. If $f$ is flat, 
  then $\Phi_{\QCOH}$ is an equivalence of abelian categories and preserves and reflects flatness.
\end{corollary}
\begin{corollary}[{\cite[Cor.~3.4.3]{MR1432058}}]\label{C:gluing_special}
  Assume that we are in the situation of Theorem \ref{T:MV-QCoh-gluing}.
  Then $M \in \QCOH(X)$ is $f$-flat if and only if $j^*M$ is 
    $f_U$-flat.
\end{corollary}
\begin{proof}
  The necessity is clear. For the sufficiency: if $j^*M$ 
  is $f_U$-flat, then $\tilde{M}=\Psi(f^*M,j^*M,\delta)$ is an $f$-flat quasi-coherent sheaf and 
  there is a natural map $\eta \colon M \to \tilde{M}$. Now $j^*\eta$ and $f^*\eta$ are 
  isomorphisms, so $\ker(\eta)$ is $f$-flat (Lemma \ref{L:supp_flat}) and $\eta$ is 
  surjective (Theorem \ref{T:MV-QCoh-gluing}\itemref{TI:MV-QCoh-gluing:surj}). So we have an exact sequence:
  \[
  0 \to \ker(\eta) \to M \to \tilde{M} \to 0
  \]
  and $\ker(\eta)$ and $\tilde{M}$ are $f$-flat. It follows that $M$ is $f$-flat, which gives 
  the sufficiency.
\end{proof}

We now consider Theorem \ref{T:MV-QCoh-gluing} in the context of algebras, or equivalently, affine schemes.
\begin{corollary}\label{C:gluing-of-algebras}
   Assume that we are in the situation of Theorem \ref{T:MV-QCoh-gluing}. 
  The natural functor
  \[
    \Phi_{\AFF,f-\mathrm{fl}}\colon
    \AFF_{f-\mathrm{fl}}(X)\to
    \AFF(X')\times_{\AFF(U')} \AFF_{f_U-\mathrm{fl}}(U),
  \]
  is an equivalence of categories. Moreover, the functor $\Phi_{\AFF}$ preserves and reflects
  \begin{enumerate}
    \item\label{CI:reflects:alg:closed-imm} closed immersions;
    \item\label{CI:reflects:alg:finite} finite morphisms;
    \item\label{CI:reflects:alg:integral} integral morphisms;
    \item\label{CI:reflects:alg:fin-type} morphisms of finite type;
  \end{enumerate}
  and the functor $\Phi_{\AFF,f-\mathrm{fl}}$ preserves and reflects
  \begin{enumerate}[resume]
    \item\label{CI:reflects:alg:fin-pres} morphisms of finite presentation.
  \end{enumerate}
\end{corollary}
\begin{proof}
  An $\Orb_X$-algebra structure on an $\Orb_X$-module $M$ is given by
  homomorphisms $\Orb_X\to M$ and $M\otimes_{\Orb_X} M\to M$ satisfying various
  compatibility conditions. If $M$ is $f$-flat, then an algebra structure on
  $\Phi_{\MOD}(M)$ descends to a
  unique algebra structure on $M$ by Corollary \ref{C:MV-QCoh-gluing:ff}.

  That $\Phi_{\AFF}$ preserves all the properties follows by definition.
  To see that $\Phi_{\AFF}$ reflects the properties, we may work fppf-locally on
  $X$ and assume that $X$ is affine and work with the categories of algebras. We let $\Phi 
  = \Phi_{\MOD}$ for the remainder of the proof.

  \itemref{CI:reflects:alg:closed-imm}--\itemref{CI:reflects:alg:finite}
  These statements follow from Theorem~\ref{T:MV-QCoh-gluing}\itemref{TI:MV-QCoh-gluing:surj}--\itemref{TI:MV-QCoh-gluing:ft}.

  \itemref{CI:reflects:alg:integral}
  Let $A\to B$ be a homomorphism of $\Orb_X$-algebras. If $\Phi(A)\to \Phi(B)$
  is integral, then $j^*A\to j^*B$ is
  integral. Thus, if $B_0$ is the integral closure of $A$ in $B$, then
  $j^*B_0=j^*B$. Write $B$ as the filtered union of finitely generated
  $B_0$-subalgebras $B_\lambda\subseteq B$. Since $j^*B_0=j^*(B_\lambda)=j^*B$,
  we have that $B/B_\lambda$ is $f$-flat;
  it follows that $\Phi(B_\lambda)\subseteq \Phi(B)$
  is a $\Phi(B_0)$-subalgebra of finite type. Thus $\Phi(B_\lambda)$ is a finite
  $\Phi(B_0)$-algebra, so $B_\lambda$ is a finite $B_0$-algebra. It follows that
  $B=\bigcup_\lambda B_\lambda$ is integral over $A$.

  \itemref{CI:reflects:alg:fin-type}
  If $A\to B$ is a homomorphism of $\Orb_X$-algebras such that
  $\Phi(A)\to \Phi(B)$ is of finite type, then write $B$ as
  a filtered union of finitely generated $A$-subalgebras $B_\lambda$. For
  sufficiently large $\lambda$, we have that $\Phi(B_\lambda)\to \Phi(B)$
  is surjective, hence so is $B_\lambda\to B$ so $A\to B$ is of finite type.

  \itemref{CI:reflects:alg:fin-pres}
  If $A\to B$ is a homomorphism of $\Orb_X$-algebras such that $\Phi(A)\to
  \Phi(B)$ is of finite presentation, then we
  have already seen that $A\to B$ is of finite type. There is an exact sequence
  $0\to I\to A[x_1,x_2,\dots,x_n]\to B\to 0$ and if $B$ is $f$-flat, then this
  sequence remains exact after applying $f^*$. If in addition $A$ is $f$-flat,
  then we conclude that $I$ is a finitely generated ideal (use
  Lemma~\ref{L:mv_bc}\itemref{LI:mv_bc:mv} and
  Theorem~\ref{T:MV-QCoh-gluing}\itemref{TI:MV-QCoh-gluing:ft}), hence that
  $A\to B$ is of finite presentation.
\end{proof}

\begin{corollary}\label{C:fmv-qaff_gluing}
  Assume that we are in the situation of Theorem \ref{T:MV-QCoh-gluing}. If $f$ is flat, 
  then $\Phi_{\AFF}$ and $\Phi_{\QAFF}$ are equivalences of categories.
\end{corollary}
\begin{proof}
  The equivalence of $\Phi_{\AFF}$ follows immediately from Corollary 
  \ref{C:gluing-of-algebras}. For $\QAFF$, we must work a little more. Some notation will 
  be useful: if $W \to Y$ is quasi-affine, then let $\bar{W} \to Y$ denote its affine hull. 
  Note that the formation of $\bar{W} \to Y$ commutes with flat base change on $Y$. 
  Similarly, for a morphism $\alpha\colon W_1 \to W_2$ of quasi-affine schemes over $Y$ 
  we let $\bar{\alpha}$ denote the induced morphism between the affine hulls. 

  Now for the faithfulness: let $\alpha$, $\beta \colon W_1 \to W_2$ be morphisms in 
  $\QAFF(X)$ such that $\Phi_{\QAFF}(\alpha) = \Phi_{\QAFF}(\beta)$. By the result for 
  $\AFF$, we see that $\bar{\alpha}=\bar{\beta}$ and the claim follows.

  Next for the fullness: consider quasi-affine $X$-schemes $W_1$ and $W_2$ and a 
  morphism $(\alpha',\alpha_U) \colon \Phi_{\QAFF}(W_1) \to \Phi_{\QAFF}(W_2)$. The result 
  for $\AFF$ implies that there is a morphism $\bar{\alpha} \colon \bar{W}_1 \to 
  \bar{W}_2$ such that $\Phi_{\AFF}(\bar{\alpha}) = (\bar{\alpha'},\bar{\alpha_U})$. It is 
  sufficient to prove that $W_1 \subseteq \bar{\alpha}^{-1}(W_2)$. But this may be 
  checked on points and $X'\amalg U \to X$ is surjective. The claim follows.

  Finally, for the essential surjectivity. Now fix 
  a triple $(W',W_U,\theta)$ in the codomain for 
  $\Phi_{\QAFF}$. This leads to a triple $(\bar{W'},\bar{W_U},\bar{\theta})$ in the 
  codomain of $\Phi_{\AFF}$ that may be glued to an affine $X$-scheme $\bar{W}$. Since 
  $U \subseteq X$ is quasi-compact and $f$ is flat and an isomorphism over $Z$, it is 
  easily verified that $X'\amalg U \to X$ is universally submersive (flatness
  is actually not needed, see Theorem~\ref{C:wMV-is-univ-submersive}).
  In particular, by base 
  changing along $\bar{W} \to X$ we may glue the quasi-compact open subsets $W' 
  \subseteq \bar{W'}$ and $W_U \subseteq \bar{W_U}$ to a quasi-compact open subset $W 
  \subseteq \bar{W}$. This proves the claim.
\end{proof}
We conclude this section with the following generalization of~\cite[Cor.~4.3]{MR0272779}.

\begin{corollary}\label{C:normalization}
  Assume that we are in the situation of Theorem \ref{T:MV-QCoh-gluing}. Let $\eta\colon \Orb_X\to j_*\Orb_U$ and 
  $\eta'\colon \Orb_{X'}\to j'_*\Orb_{U'}$ denote the unit maps.
  \begin{enumerate}
  \item $\eta$ is injective if and only if $\eta'$ is injective.
  \item $\eta$ is integrally closed if and only if $\eta'$ is integrally closed.
  \item If $\overline{X}$ denotes the integral closure of $X$ in $U$, i.e.,
    $\spec_X(\mathcal{A})$ where $\mathcal{A}$ is the integral closure of
    $\Orb_X$ with respect to $\eta$, then $\overline{X}':=\overline{X}\times_X
    X'$ is the integral closure of $X'$ in $U'$ and the square
    of $U\to \overline{X}$ and $\overline{X}'\to \overline{X}$
    is a tor-independent Mayer--Vietoris square.
  \end{enumerate}
\end{corollary}
\begin{proof}
  By Corollary \ref{C:gluing-of-algebras}\itemref{CI:reflects:alg:closed-imm},
  there is a bijection of partially ordered sets
  \[
    \Phi\colon \Cl_{f-\mathrm{fl}}(X)\to \Cl(X')\times_{\Cl(U')} \Cl_{f_U-\mathrm{fl}}(U),
  \]
  where $\Cl(X)$ denotes the set of closed substacks $V\inj X$ and
  $\Cl_{f-\mathrm{fl}}$ denotes the subset of closed substacks such that
  $\Orb_V$ is $f$-flat. If we let $\overline{U}$ and
  $\overline{U'}$ denote the schematic closures of $U$ and $U'$ in $X$ and
  $X'$ respectively, then $\overline{U}$ is $f$-flat (Corollary \ref{C:gluing_special}) and $\overline{U}$
  corresponds to a triple $(\overline{U}\times_X X',U',U)$ on the right hand
  side. But $\overline{U}$ is minimal among the closed substacks of $X$ that
  contains $U$ and $\overline{U'}$ is minimal among the closed substacks
  of $X'$ that contains $U'$. It follows that 
  $\Phi(\overline{U})=(\overline{U'},U',U)$. Thus, $X=\overline{U}$ if and only
  if $X'=\overline{U'}$. Equivalently, $\eta$ is injective if and only if
  $\eta'$ is injective.

  Similarly, Corollary \ref{C:gluing-of-algebras}\itemref{CI:reflects:alg:integral} induces an equivalence of categories of
  integral morphisms
  \[
    \Phi\colon \Int_{f-\mathrm{fl}}(X)\to \Int(X')\times_{\Int(U')} \Int_{f_U-\mathrm{fl}}(U).
  \]
  If we let $\Int(X,U)$ denote the integral morphisms $W\to X$ such that
  $W|_U\to U$ is an isomorphism and $U$ is schematically dense in $W$, then
  $\Int(X,U)$ is equivalent to the bounded lattice of sub-$\Orb_X$-algebras of
  $j_*\Orb_U$ that are integral over $\Orb_X$. These extensions are
  automatically $f$-flat, since they are $f_U$-flat after restricting to
  $\Orb_U$ (Corollary \ref{C:gluing_special}). We thus obtain a bijection of bounded lattices:
  \[
    \Phi\colon \Int(X,U)\to \Int(X',U').
  \]
  Indeed, the only non-obvious detail is that $U'$ is schematically dense
  in $\Phi(W)=W\times_X X'$ and that $U$ is schematically dense in
  $\Phi^{-1}(W',U',U)$. This follows from the
  previous part since the square 
  \[
  \vcenter{\xymatrix{
      W\times_X U'\ar[r]\ar[d] & W\times_X X'\ar[d] \\
      W\times_X U\ar[r]& W\ar@{}[ul]|\square
    }}
  \]
  is a tor-independent Mayer--Vietoris square (Lemma \ref{L:mv_bc}\itemref{LI:mv_bc:mv}).
  Moreover, the minimal elements of these lattices are $\overline{U}$ and $\overline{U'}$.
  and the maximal elements are $\overline{X}$ and $\overline{X'}$. The result
  follows.
\end{proof}

\section{\'Etale sheaves of sets on stacks}\label{S:etale-sheaves}
In this section we generalize some fundamental results on constructible sheaves
in SGA4 from schemes to algebraic stacks.

Let $X$ be an algebraic stack. We let $\Et(X)$ denote the category of \'etale
representable morphisms $E\to X$. We identify $\Et(X)$ with the category of
cartesian lisse-\'etale sheaves of sets. Under this identification
finitely presented \'etale morphisms correspond to constructible sheaves
of sets.

If $X$ is a quasi-compact and quasi-separated algebraic space or
Deligne--Mumford stack, then there is an \'etale presentation by an affine
scheme. Using this presentation it is easily seen that every \'etale sheaf on
$X$ is a filtered colimit of constructible sheaves. We will now extend this
result to every quasi-compact and quasi-separated algebraic stack.

Recall that if $f\colon X\to Y$ is flat of finite presentation with
geometrically reduced fibers, then there exists a factorization $X\to
\pi_0(X/Y)\to Y$ where the first map has connected fibers and the second
is representable and \'etale~\cite[Thm.~2.5.2]{MR2820394}. This construction
commutes with arbitrary base change on $Y$ and is functorial in $X$.
The following result is due to J.\ Wise.

\begin{proposition}[{\cite[Thm.~4.5]{wise_logmaps}}]
Let $f\colon X\to Y$ be flat of finite presentation with geometrically reduced
fibers (e.g., $f$ smooth, quasi-compact and quasi-separated). If every
\'etale sheaf on $X$ is a filtered colimit of constructible sheaves
(e.g., $X$ is
a quasi-compact and quasi-separated algebraic space), then $f^*\colon
\Et(Y)\to \Et(X)$ admits a left-adjoint $f_!\colon \Et(X)\to \Et(Y)$ with
the following properties:
\begin{enumerate}
\item If $(E\to X)\in \Et(X)$, then the unit induces an $X$-morphism
  $E\to f^*f_!E$.
  This gives a factorization $E\to f_!E\to Y$ of the morphism $E\to X\to Y$
  such that $E\to f_!E$ has geometrically connected fibers.
\item $f_!$ preserves constructible sheaves.
\item $f_!$ commutes with pull-back: $g^*f_!=f'_!g'^*$ for any morphism
  $g\colon Y'\to Y$, where $f'\colon X':=X\times_Y Y'\to Y'$ and
  $g'\colon X'\to X$.
\end{enumerate}
\end{proposition}
\begin{proof}
For constructible sheaves, it is readily seen that $f_!(E\to X):=(\pi_0(E/Y)\to
Y)$ is a left adjoint of $f^*$ and it commutes with arbitrary base change.  It
remains to extend the construction to non-constructible \'etale sheaves $E\to
X$. If $E=\varinjlim E_\lambda$ is a filtered colimit of constructible sheaves,
then necessarily $f_!E=\varinjlim f_!E_\lambda$.
\end{proof}

We may now generalize \cite[Exp.~IX, Cor.~2.7.2, Prop.~2.14]{MR0354654}
and~\cite[Exp.~XII, Prop.~6.5 (i)]{MR0354654} to quasi-compact and
quasi-separated algebraic stacks.

\begin{proposition}\label{P:qcqs-cons-colimit}
Let $X$ be a quasi-compact and quasi-separated algebraic stack. Then
every \'etale sheaf of sets is a filtered colimit of constructible
sheaves.
\end{proposition}
\begin{proof}
The result is known for affine schemes (and quasi-compact and quasi-separated
schemes). Pick a smooth presentation $p\colon U\to X$ with $U$ affine. Let
$F\to X$ be an \'etale sheaf. Choose an epimorphism $\coprod_{i\in I} G'_i\to
p^*F$ where the $G'_i$ are constructible. Let $G_i=p_!G'_i$ which is a
constructible sheaf. Then $\coprod G_i=p_!(\coprod_i G'_i)\to F$ is an
epimorphism since $\coprod_i G'_i\to p^*p_!(\coprod_i G'_i)\to p^*F$ is an
epimorphism.

The remainder of the proof is standard, cf.~\cite[Exp.~IX,
  Cor.~2.7.2]{MR0354654}. For every finite subset $J\subseteq I$, the coproduct
$G_J=\coprod_{i\in J} G_i$ is constructible. The fiber product
$H_J:=G_J\times_F G_J$ is not constructible but at least quasi-separated since
it is a subsheaf of the constructible sheaf $G_J\times_X G_J$. Consider the
set $\Lambda$ of pairs $(J,H')$ where $J\subseteq I$ is finite and $H'\subseteq
H_J$ is quasi-compact, and hence constructible. For $\lambda=(J,H')\in
\Lambda$, let $F_\lambda=\coker(\equalizer{H'}{G_J})$ which is a constructible sheaf. We
order $\Lambda$ by $(J_1,H'_1)\leq (J_2,H'_2)$ if $J_1\subseteq J_2$ and
$g(H'_1)\subseteq H'_2$ where $g\colon H_{J_1}\to H_{J_2}$. Then
$F=\varinjlim_{\lambda\in\Lambda} F_\lambda$ is a filtered colimit of
constructible sheaves.
\end{proof}

\begin{proposition}\label{P:mono-decomposed}
Let $X$ be a quasi-compact and quasi-separated algebraic stack. Let
$F\in \Et(X)$ be a constructible sheaf of sets. Then there exist finite morphisms
$p_i\colon X'_i\to X$, $i=1,2,\dots,n$ and finite sets $A_1,A_2,\dots,A_n$ and a
monomorphism $F\inj \prod (p_i)_* \underline{A_i}_{X'_i}$.
\end{proposition}
\begin{proof}
There exists a stratification of $X$ into locally closed constructible
substacks $Y_i$ such that $F|_{Y_i}$ is locally
constant~\cite[Prop.~4.4]{MR2774654}. If $u_i\colon Y_i\to X$ denotes the
corresponding quasi-compact immersion, then $F\to \prod
(u_i)_*(u_i)^*F$ is a monomorphism. After refining the stratification, we can
assume that the cardinality of $F|_{Y_i}$ is constant. Let $q_i\colon Y'_i\to
Y_i$ be a finite \'etale surjective morphism such that $q_i^*u_i^*F$ is a
constant sheaf with value $A_i$.

Let $X_i$ be the closure of $Y_i$ and let $p_i\colon X'_i\to X$ be the
integral closure of $X$ with respect to $Y'_i\to Y_i\to X_i\to X$. Then $p_i$ is
integral and $p_i|_{Y_i}=q_i$.  If $v_i\colon Y'_i\to X'_i$ denotes the open
immersion, then $(v_i)_*\underline{A_i}_{Y_i'}=\underline{A_i}_{X_i'}$ is
constant. Thus,
\[
F\to \prod (u_i)_*(u_i)^*F
  \inj \prod (u_i)_*(q_i)_*(q_i)^*(u_i)^*F
  = \prod (p_i)_*\underline{A_i}_{X'_i}
\]
is a monomorphism.

Finally, write $X'_i\to X_i$ as an inverse limit of finite
morphisms~\cite{rydh-2014}. By an easy limit argument, we can replace
$p_i$ by a finite morphism.
\end{proof}

For an algebraic stack $X$, we let $\clopen(X)$ denote the boolean algebra of
closed and open substacks.

\begin{proposition}\label{P:H0-bij-vs-clopen}
Let $h\colon Y\to X$ be a morphism of
algebraic stacks. If $X$ is quasi-compact and quasi-separated, then the following conditions are equivalent.
\begin{enumerate}
\item\label{PI:H0-bij}
For every sheaf of sets $F\in \Et(X)$, the canonical map
\[
H^0(X,F)\to H^0(Y,h^*F)
\]
is bijective.
\item\label{PI:H0-bij:cons}
Condition~\itemref{PI:H0-bij} for constructible sheaves.
\item\label{PI:clopen-bij}
For every finite morphism $f\colon X'\to X$, the canonical map
\[
\clopen(X')\to \clopen(Y\times_X X')
\]
is bijective.
\end{enumerate}
\end{proposition}
\begin{proof}
The equivalence between~\itemref{PI:H0-bij} and~\itemref{PI:H0-bij:cons}
follows by
Proposition~\ref{P:qcqs-cons-colimit}. That~\itemref{PI:H0-bij} implies~\itemref{PI:clopen-bij} follows by the
following two observations: (a) if $A$ is a two-point set, then $H^0(X,f_*\underline{A}_{X'})=\clopen(X')$, and (b) by finite base change $h^*f_*\underline{A}_{X'}=(f_Y)_*\underline{A}_{Y\times_X X'}$.

To see that~\itemref{PI:clopen-bij} implies~\itemref{PI:H0-bij}, take a monomorphism $F\inj G$ as in
Proposition~\ref{P:mono-decomposed}. Then by~\itemref{PI:clopen-bij}, $H^0(X,G)\to H^0(Y,h^*G)$
is bijective. It follows that $H^0(X,F)\to H^0(Y,h^*F)$ is injective.
Finally, take $H=G\amalg_F G$. Then we have a diagram
\[
\xymatrix{
H^0(X,F)\ar[r]\ar[d] & H^0(X,G)\ar@<.5ex>[r]\ar@<-.5ex>[r]\ar[d] & H^0(X,H)\ar[d]\\
H^0(Y,h^*F)\ar[r] & H^0(Y,h^*G)\ar@<.5ex>[r]\ar@<-.5ex>[r] & H^0(Y,h^*H)
}
\]
with exact rows and injective vertical maps and bijective middle map.
It follows that the left map is bijective.
\end{proof}
We recall the following well-known definition.
\begin{definition}[Henselian pairs]\label{D:henselian}
A pair of algebraic stacks $(X,X_0)$ is
a \emph{henselian pair} if $i\colon X_0\inj X$ is a closed immersion and for
every finite morphism $X'\to X$, the natural map
\[
\clopen(X')\to\clopen(X'\times_X X_0)
\]
is bijective.
\end{definition}
We have the following simple lemma.
\begin{lemma}\label{L:hens_p_int}
Let $(X,X_0)$ be a henselian pair. Let $X'\to X$ be an integral morphism. If $X$ is quasi-compact and quasi-separated, then $(X',X'\times_X X_0)$ is a henselian pair.
\end{lemma}
\begin{proof}
  Since $X'\to X$ is a limit of
finite morphisms~\cite{rydh-2014}, the result follows from a simple approximation argument.
\end{proof}

\begin{remark}[Proper base change]\label{R:proper_bc_et}
  Let $(X,X_0)$ be a henselian pair, where $X$ is quasi-compact and quasi-separated. If 
  $g\colon X'\to X$ is \emph{proper and representable},
  then $(X',X'\times_X X_0)$ is a henselian pair (see
  \cite[Cor.~B.4]{MR3148551} and \cite[Cor.~1]{MR1286833}). This
  follows from the existence of the Stein factorization
  $X'\to \spec_{\Orb_X} g_*\Orb_{X'}\to X$ where the first map is
  proper with geometrically connected fibers and the second map is
  integral \cite[\spref{0A1C}]{stacks-project}. This is a baby case of
  the proper base change theorem in \'etale cohomology.
\end{remark}

\section{Mayer--Vietoris squares in \'etale cohomology}\label{S:etale-gluing}
We will now glue \'etale morphisms, or equivalently, \'etale sheaves of
sets. It is thus natural to introduce the following squares which are
analogous to Mayer--Vietoris $\DQCOH$-squares.
  \begin{definition}
    Fix a cartesian square as in \eqref{E:MV-square}. It is a Mayer--Vietoris $\Et$-square if 
    the following conditions are satisfied:
    \begin{enumerate}
    \item the natural transformation $f^*j_*\to j'_*f_U^*$
      is an isomorphism for every cartesian sheaf of sets $F\in\Et(U)$; and
    \item $f^*\colon \Et_{Z}(X)\to\Et_{Z'}(X')$ is 
      an equivalence of categories, where $\Et_Z(X) = \{F \in \Et(X) \suchthat j^*F = 0\}$ and similarly 
      for $\Et_{Z'}(X')$.
    \end{enumerate}
  \end{definition}
  Note that $\Et_Z(X)$ does not depend on the choice of $Z$ and that
  $i_*\colon \Et(Z)\to \Et_Z(X)$ is an equivalence.

  For Mayer--Vietoris $\Et$-squares, gluing is immediate from recollement.

  \begin{theorem}\label{T:gluing-of-etale}
    Consider a Mayer--Vietoris $\Et$-square. Then the functor
    \[
    \Phi_\Et\colon \Et(X)\to \Et(X')\times_{\Et(U')} \Et(U)
    \]
    is an equivalence of categories.
    %% For $j$ quasi-compact flat monomorphism,
    %% we have an equivalence of constructible \'etale sheaves.
  \end{theorem}
  \begin{proof}
    By recollement~\cite[Exp.~IV, Thm.~9.5.4]{MR0354652},
    \[
    \Et(X)\cong (\Et(Z),\Et(U),i^*j_*),
    \]
    that is,
    the category $\Et(X)$ is equivalent to the category of triples
    $E_Z\in \Et(Z)$, $E_U\in \Et(U)$, $\psi\colon E_Z\to i^*j_*E_U$.
    Similarly,
    \[
    \Et(X')\cong (\Et(Z'),\Et(U'),i'^*j'_*)
    \]
    and
    \begin{align*}
    \Et(X')\times_{\Et(U')} \Et(U) &\cong (\Et(Z'),\Et(U),i'^*j'_*f_U^*) \\
     &\cong (\Et(Z),\Et(U),(f_Z)_*i'^*j'_*f_U^*)
    \end{align*}
    where we have used that $(f_Z)_*$ is an equivalence of categories.
    Since $(f_Z)_*i'^*j'_*f_U^*=(f_Z)_*i'^*f^*j_*=i^*j_*$ the result follows.
    %% If $j$ is a quasi-compact flat monomorphism, then it is the inverse limit
    %% of quasi-compact open immersions $j_\lambda\colon U_\lambda\to U$.
    %% The right-hand side is then the direct limit of the categories
    %% $\Et(X')\times_{\Et(U'_\lambda)} \Et(U_\lambda)$ and the result follows.
  \end{proof}

  We will now proceed to show that weak Mayer--Vietoris squares are
  Mayer--Vietoris $\Et$-squares. We begin with the following result that
  generalizes~\cite[Cor.~4.4]{MR0272779}.
  \begin{proposition}\label{P:wMV-clopen}
    Fix a weak Mayer--Vietoris square as in \eqref{E:MV-square}. Assume that 
    $(X,Z)$ and $(X',Z')$ are henselian pairs. If $X$, $X'$, $U$ and $U'$ are all quasi-compact and 
    quasi-separated, then the natural map
    \[
    f_U^*\colon \clopen(U)\to \clopen(U')
    \]
    is bijective.
  \end{proposition}
  \begin{proof}
    Since $X$ and $U$ are quasi-compact and quasi-separated, we may assume that the complement $i\colon Z \inj X$ is finitely presented \cite[Prop.~8.2]{rydh-2014}. Thus, we may replace the square with its blow-up so that it becomes a tor-independent 
    Mayer--Vietoris square (Lemma~\ref{L:blow-up-is-MV-square}). Note
    that $(X',Z')$ and $(X,Z)$ remain henselian pairs (Remark \ref{R:proper_bc_et}).
    
  By Corollary~\ref{C:normalization}, we may replace $X$ and $X'$ by
  $\overline{X}$ and $\overline{X'}$ and assume that $X$ and $X'$ are
  integrally closed with respect to $U$ and $U'$ respectively. Since the open
  and closed subsets of an algebraic stack $W$ are in bijection with
  idempotents of $\Gamma(W,\Orb_W)$, it follows that $\clopen(X)\to \clopen(U)$
  and $\clopen(X')\to \clopen(U')$ are bijections. The corollary thus follows from
  the commutativity of the following diagram:
  \[
  \begin{gathered}[b]
  \xymatrix{
  \clopen(Z)\ar@{=}[d] & \clopen(X)\ar[l]_{\cong}\ar[r]^{\cong}\ar[d] & \clopen(U)\ar[d] \\
  \clopen(Z) & \clopen(X')\ar[l]_{\cong}\ar[r]^{\cong} & \clopen(U').}\\[-\dp\strutbox]
  \end{gathered}
  \qedhere
  \]
  \end{proof}

  We can now prove Gabber's rigidity theorem. For affine henselian
  pairs, this is proven in~\cite[Exp.~20, Thm.~2.1.1]{MR3309086}. See
  Remark~\ref{R:rigidity} for some history of this result.

  \begin{theorem}[Rigidity theorem]\label{T:rigidity-theorem}
    Fix a weak Mayer--Vietoris square as in \eqref{E:MV-square}. Assume that 
    $(X,Z)$ and $(X',Z')$ are henselian pairs. If $X$, $X'$, $U$ and $U'$ are all 
    quasi-compact and quasi-separated, then the natural map:
    \[
    H^0(U,F)\to H^0(U',f_U^*F)
    \]
    is a bijection for all sheaves of sets $F\in\Et(U)$.
  \end{theorem}
  \begin{proof}
    It is enough to prove that $\clopen(V)\to \clopen(U'\times_U V)$ is
    bijective for every finite morphism $V\to U$
    (Proposition~\ref{P:H0-bij-vs-clopen}). By
    Zariski's main theorem~\cite[Thm.~8.1]{rydh-2014}, we can extend the finite
    morphism $V\to U$ to a finite morphism $\overline{V}\to X$. Since weak
    Mayer--Vietoris squares are stable under arbitrary base change (Lemma~    \ref{L:mv_bc}\itemref{LI:mv_bc:w_f_mv}),
    it is enough to prove
    that $\clopen(U)\to \clopen(U')$ is bijective, which is Proposition~\ref{P:wMV-clopen}.
  \end{proof}

  \begin{corollary}\label{C:wMV-is-MV-Et}
    Fix a weak Mayer--Vietoris square as in \eqref{E:MV-square}. If $j$ is quasi-compact, 
    then it is a Mayer--Vietoris $\Et$-square. 
  \end{corollary}
  \begin{proof}
    We need to verify that the natural morphism $f^*j_*\to j'_*f_U^*$ is
    an isomorphism. This equality certainly holds
    over $U'$ since the counits of the adjunctions $(j^*,j_*)$ and $(j'^*,j'_*)$
    are isomorphisms and hence $j'^*j'_*f_U^*=f_U^*$ and
    $j'^*f^*j_*=f_U^*j^*j_*=f_U^*$. It is thus enough to verify the equality
    over points of $Z$. We can first assume that $X$ is affine by working
    smooth-locally on $X$ and then replace $X$ with the henselization at a
    point $z\in Z$. Then $X'$ is Deligne--Mumford in a neighborhood of $Z$ and
    we can thus replace $X'$ with the henselization at $z\in Z$; in particular, $X$ and $X'$ 
    are quasi-compact and quasi-separated. Then
    the equality $f^*j_*=j'_*f_U^*$ becomes $H^0(U,F)=H^0(U',f_U^*F)$, 
    which follows by the rigidity theorem.
  \end{proof}
  We can now prove Theorem \ref{MT:etale-gluing-for-wmv}.
  \begin{proof}[Proof of Theorem \ref{MT:etale-gluing-for-wmv}]
    Combine Corollary \ref{C:wMV-is-MV-Et} with Theorem \ref{T:gluing-of-etale}.
  \end{proof}

  \begin{corollary}\label{C:wMV-is-univ-submersive}
    Fix a weak Mayer--Vietoris square as in \eqref{E:MV-square}. If $j$ is quasi-compact,  
    then $X'\amalg U\to X$ is universally submersive and
    $|X|=|X'|\amalg_{|U'|} |U|$ is a pushout of topological spaces.
  \end{corollary}
  \begin{proof}
    Since weak Mayer--Vietoris squares are preserved under arbitrary base
    change it is enough to prove the latter statement. Set-theoretically,
    $|X|=|X'|\amalg_{|U'|} |U|$ holds since $f_Z \colon Z' \to Z$ is an isomorphism.
    It is thus enough to prove that $|X|$ has the correct topology. Now a morphism of stacks is an open immersion if and only if it is an \'etale
    monomorphism. That an \'etale morphism is a monomorphism can be checked
    pointwise; thus, we have a bijection
    \[
    \Phi_{\Op}\colon \Op(X)\to \Op(X')\times_{\Op(U')} \Op(U).
    \]
    It follows that a subset $W\subseteq |X|$ is open if and only if
    $j^{-1}(W)$ and $f^{-1}(W)$ are open.
  \end{proof}

  \begin{remark}\label{R:rigidity}
    The rigidity theorem holds more generally for cohomology as well. Fix a weak Mayer--Vietoris square as in \eqref{E:MV-square}
    and assume
    that $(X,Z)$ and $(X',Z')$ are affine henselian pairs. If $n=0$
    (resp.\ $n\leq 1$, resp.\ $n$ an integer), then
    \[
    H^n(U,F)\to H^n(U',F)
    \]
    is a bijection for all sheaves of sets $F\in\Et(U)$ (resp.\ sheaves of
    ind-finite groups, resp.\ sheaves of torsion abelian groups). When $X$ is
    noetherian, this is Gabber--Fujiwara's rigidity
    theorem~\cite[Cor.~6.6.4]{MR1360610}. For $n=0,1$, this was extended
    to non-noetherian schemes by Gabber~\cite[Thm.~7.1]{gabber-2005a},
    cf.\ \cite[Exp.~20, Thm.~2.1.1]{MR3309086}. For $n\geq 2$, the
    non-noetherian case is
    sketched by Gabber in~\cite[Exp.~20, \S4.4, CTC]{MR3309086}.  Note that
    the general case reduces to the case where $X'$ is the completion of $X$
    in $Z$. Indeed, such a completion is a weak Mayer--Vietoris square and
    the completions
    of $X$ in $Z$ and $X'$ in $Z$ are equal by definition.
  \end{remark}  

\section{Gluing of algebraic spaces along Mayer--Vietoris squares}\label{S:algsp-gluing}
In this section, we prove the main theorems of the article. We begin with a slight strengthening of Theorem \ref{MT:timv_dm_push}.

\begin{proposition}\label{P:cocartesian-in-cat-of-quasi-DM}
Fix an algebraic stack $S$ and a tor-independent Mayer--Vietoris square as in \eqref{E:MV-square} over $S$ with $j$ quasi-compact. Let $W\to S$ be an algebraic stack. Then
\[
\Phi_{\Hom_S(-,W)}\colon \Hom_S(X,W)\to \Hom_S(X',W)\times_{\Hom_S(U',W)} \Hom_S(U,W)
\]
is fully faithful. If either
\begin{enumerate}
\item \label{PI:cocartesian-in-cat-of-quasi-DM:DM} $W\to S$ is Deligne--Mumford;
\item \label{PI:cocartesian-in-cat-of-quasi-DM:qfd} $\Delta_{W/S}$ is quasi-finite and $\Delta_{\Delta_{W/S}}$ is a quasi-compact
  immersion; or
\item \label{PI:cocartesian-in-cat-of-quasi-DM:lqfsd} $\Delta_{W/S}$ is locally quasi-finite and separated;
\end{enumerate}
then $\Phi_{\Hom_S(-,W)}$ is an equivalence of groupoids.
In particular, the square is cocartesian in the category
of Deligne--Mumford stacks.
\end{proposition}
\begin{proof}
The question is fppf-local on $X$, so we may assume that $X$ is affine. We
may also replace $S$ and $W$ with $X$ and $W\times_S X\to X$ and assume that
$X=S$. Further, we may replace $X'$ with a quasi-compact open neighborhood
of $Z$.
Then, we may also assume that $W$ is quasi-compact.

If $W\to X$ is arbitrary (resp.\ representable, resp.\ a
monomorphism), then $\Delta_{W/X}$ is representable (resp.\ a monomorphism,
resp.\ an isomorphism). Fully faithfulness of $\Phi_{\Hom_X(-,W)}$ follows if
$\Phi_{\Hom_X(-,W\times_{W\times_X W} X)}$ is an equivalence for
every morphism $X\to W\times_X W$. By induction on the diagonal,
we may thus assume that $\Phi_{\Hom_X(-,W)}$ is fully faithful and
it is enough to prove that $\Phi_{\Hom_X(-,W)}$ is essentially surjective
when \itemref{PI:cocartesian-in-cat-of-quasi-DM:DM},
\itemref{PI:cocartesian-in-cat-of-quasi-DM:qfd} or
\itemref{PI:cocartesian-in-cat-of-quasi-DM:lqfsd} holds.

% By~\cite[Prop.~4.1]{etale_dev_add}... no non-representable :(
If $W$ is Deligne--Mumford, then there exists an \'etale presentation
$W'\to W$. If $W$ is as in
\itemref{PI:cocartesian-in-cat-of-quasi-DM:qfd} or
\itemref{PI:cocartesian-in-cat-of-quasi-DM:lqfsd}, then
by~\cite[Thm.~7.2]{MR2774654} or~\cite[Prop.~6.11]{MR3084720}
there exist an \'etale representable morphism $W'\to W$ and a finite
faithfully flat morphism $V\to W'$ such that $V$ is affine.

Given maps $U\to W$ and $X'\to W$ that agree on $U'$, we obtain, by pulling
back $W'\to W$, an element of $\Et(X')\times_{\Et(U')} \Et(U)$, hence a unique element of $(E\to
X)\in \Et(X)$ by Corollary~\ref{C:wMV-is-MV-Et} and
Theorem~\ref{T:gluing-of-etale}. Pulling-back the square along
$E\to X$, we may replace $X$ by $E$ and assume that $W=W'$ in all three cases.

In the latter two cases, we additionally pull-back $V\to W'=W$ to finite flat morphisms
over $X'$, $U'$ and $U$. These glue to a unique finite faithfully flat morphism
$F\to X$ (Corollary \ref{C:gluing-of-algebras} and Theorem \ref{T:MV-QCoh-gluing}\itemref{TI:MV-QCoh-gluing:flat}). We may thus replace $X$ with $F$ and assume that $V=W=W'$ are affine schemes.

Let
$A_W=\Gamma(W,\Orb_W)$ for any algebraic stack $W$. The
map $\Phi_{\Hom(-,W)}$ then becomes
\begin{align*}
\Hom(A_W,A_X) &\to
  \Hom(A_W,A_{X'})\times_{\Hom(A_W,A_{U'})} \Hom(A_W,A_U) \\
&= \Hom(A_W,A_{X'}\times_{A_{U'}}A_U).
\end{align*}
This is an isomorphism since $A_X\to A_{X'}\times_{A_{U'}}A_{U}$ is an
isomorphism by Corollary~\ref{C:MV-QCoh-gluing:ff} applied to the structure sheaf
$\Orb_X$.
\end{proof}
We can now prove Theorem~\ref{MT:timv_dm_push}.
\begin{proof}[Proof of Theorem~\ref{MT:timv_dm_push}]
This is the last statement of Proposition~\ref{P:cocartesian-in-cat-of-quasi-DM}.
\end{proof}
We can now also generalize Corollary \ref{C:MV-QCoh-gluing:ff}
from quasi-coherent sheaves to algebraic spaces.
\begin{corollary}\label{C:gluing-of-maps-of-alg-spaces}
Fix a tor-independent Mayer--Vietoris square as in \eqref{E:MV-square} with $j$ quasi-compact. Let $Y\to X$ and $Z\to X$ be relatively
Deligne--Mumford (e.g., representable). If $Y\to X$ is $f$-flat,
then
\[
\Hom_X(Y,Z)\to \Hom_X(Y\times_X X',Z)\times_{\Hom_X(Y\times_X U',Z)} \Hom_X(Y\times_X U,Z)
\]
is bijective.
\end{corollary}
\begin{proof}
Since $Y\to X$ is $f$-flat, the pull-back of the square along $Y\to X$ is
a tor-independent Mayer--Vietoris square (Lemma \ref{L:mv_bc}\itemref{LI:mv_bc:mv}).
The result thus follows from
Proposition~\ref{P:cocartesian-in-cat-of-quasi-DM}.
\end{proof}
We have now proved Theorem \ref{MT:glue_timv} in its entirety.
\begin{proof}[Proof of Theorem \ref{MT:glue_timv}]
  Claim \itemref{MTI:glue_timv:qcoh} is Theorem \ref{T:MV-QCoh-gluing} and claim 
  \itemref{MTI:glue_timv:algsp} is Corollary \ref{C:gluing-of-maps-of-alg-spaces}.
\end{proof}

\begin{remark}
The map in Corollary~\ref{C:gluing-of-maps-of-alg-spaces} need not be
injective if the square is a weak Mayer--Vietoris square. Indeed,
Example~\ref{E:weak-but-not-MV} is an example of a weak Mayer--Vietoris square
such that $\Gamma(X)\to \Gamma(X')\times_{\Gamma(U')}\Gamma(U)$ is not
injective. If $f,g\in \Gamma(X)$ are two element that have equal images, then
the corresponding maps $f,g\colon X\to \Aff^1$ become equal after restricting
to $X'$ and $U$.
\end{remark}
We can also now prove Theorem \ref{MT:glue_fmv}.
\begin{proof}[Proof of Theorem \ref{MT:glue_fmv}]
That $\Phi_{\QCOH}$ is an equivalence is Corollary \ref{C:MV-QCoh-gluing:flatness}. That $\Phi_{\AFF}$ and $\Phi_{\QAFF}$ are equivalences is Corollary \ref{C:fmv-qaff_gluing}. That $\Phi_{\AlgSp}$ is fully faithful is a special case of
Corollary~\ref{C:gluing-of-maps-of-alg-spaces}. That $\Phi_{\Hom(-,W)}$ is fully faithful for every algebraic stack $W$ is Proposition \ref{P:cocartesian-in-cat-of-quasi-DM}. That $\Phi_{\Hom(-,W)}$ is an equivalence when $W$ has quasi-affine diagonal follows from Corollary \ref{C:fmv-qaff_gluing} and an identical argument to \cite[Cor.~6.5.1(a)]{MR1432058}.

It remains to prove \itemref{MTI:glue_fmv:algsp-exc}: $\Phi_{\AlgSp_{\lfp}}$
is an equivalence when $X$ is locally the spectrum of a $G$-ring.
For quasi-separated algebraic
spaces the essential surjectivity of $\Phi_{\AlgSp_{\lfp,\qs}}$
follows as in~\cite[Thm.~5.2 (ii), Cor.~5.6 (iii), Thm.~5.7]{MR1432058} but
since we are working in a slightly more general setting let us write out the
details. For brevity, we let $\Phi=\Phi_{\AlgSp_{\lfp}}$.

Since algebraic spaces satisfy descent for the fppf topology, we may use Proposition \ref{P:qaff_dom} and the \'etale gluing result \cite[Thm.~A]{MR2774654}, and so assume that $X$ is affine, the spectrum of a $G$-ring,
and $X'$ is quasi-affine.

If $P$ is a property of morphisms of algebraic spaces, then we say that a
morphism of triples is $P$ if the three components are $P$. Since $X'\amalg
U\to X$ is faithfully flat and quasi-compact, a morphism $f\colon W_1\to W_2$ in
$\AlgSp_{\lfp}(X)$ is quasi-compact (resp.\ quasi-separated, resp.\ \'etale,
resp.\ open, resp.\ a monomorphism) if and only if $\Phi(f)$ has the same
property~\cite[IV.2.7.1, IV.17.7.3 (ii)]{EGA}.

We now prove essential surjectivity of $\Phi$. Thus, consider a triple $W'\to X'$, $W_U\to
 U$, $W_{U'}\to U'$ of algebraic spaces, locally of finite presentation.

We will begin by showing that it
is enough to prove essential surjectivity of $\Phi$ for the subcategories of
\emph{quasi-compact} algebraic spaces
(cf.~\cite[Thm.~5.7]{MR1432058}). Write $W'$ and $W_U$ as filtered
unions of
quasi-compact open subspaces $W'_\lambda$ and $W_{U,\mu}$ respectively. Since $j'\colon U' \to X'$ is quasi-compact, for
every $\lambda$, the open subspace $W'_\lambda\cap W_{U'}$ is
quasi-compact. Hence, for sufficiently large $\mu=\mu(\lambda)$, the inverse
image $W_{U',\mu}:=f_U^{-1}(W_{U,\mu})$ contains $W'_\lambda\cap W_{U'}$. We
may thus form the triple $(W'_\lambda\cup W_{U',\mu},W_{U,\mu},W_{U',\mu})$ of
quasi-compact algebraic spaces. By assumption, this triple is in the
essential image of $\Phi$ and descends to an algebraic space $W_{\lambda,\mu}$.
We then let $W=\bigcup_{\lambda,\mu} W_{\lambda,\mu}$ where the union runs over
all $\lambda$ and $\mu\geq \mu(\lambda)$.

We next assume that the triple is quasi-compact and
quasi-separated. In this case, we claim that we are free to replace $X$ with
any flat covering
$(X_i\to X)$ such that every $X_i\to X$ is a filtered limit of flat and finitely
presented morphisms $X_{i,\lambda}\to X$. Indeed, assume that the result holds
for the $X_i$, that is, there exists an algebraic space $W_i\to X_i$ of finite
presentation such that $\Phi(W_i)\cong (W',W_{U'},W_U)\times_X X_i$.  Then, by
standard limit arguments, there is for every $i$ and every sufficiently large
$\lambda=\lambda(i)$ an algebraic space $W_{i,\lambda}\to X_{i,\lambda}$ of
finite presentation such that $\Phi_{X_{i,\lambda}}(W_{i,\lambda})\cong
(W',W_{U'},W_U)\times_X X_{i,\lambda}$. Since $\Phi$ is fully faithful over
$X_{i,\lambda}\times_X X_{i,\lambda}$ and $X_{i,\lambda}\times_X
X_{i,\lambda}\times_X X_{i,\lambda}$ there is a canonical gluing datum for
$W_{i,\lambda}\to X_{i,\lambda}$ along $X_{i,\lambda}\to X$ which is flat and
of finite presentation.  So by fppf descent, $W_{i,\lambda}\to X_{i,\lambda}$
descends to an algebraic space over the open image of $X_{i,\lambda}\to
X$. Since we can find a finite number of such $X_{i,\lambda}$ that cover $X$,
the claim follows.

Since $X$ is the spectrum of a $G$-ring, the completion map $\widehat{X}_x\to X$ is a regular
morphism. Hence, by Popescu's theorem \cite{MR818160}, it is a limit of smooth morphisms. Since
$(\widehat{X}_x\to X)_{x\in X}$ is a flat cover, we may replace $X$ with
$\widehat{X}_x$ for some $x$ and assume that $X$ is the spectrum of a complete local ring. The completion of $X'$ at $z$ equals the completion of $X$ at $z$; hence,
$X'\to X$ has a section $s\colon X\to X'$. By Lemma~\ref{L:sections-give-MV},
this gives rise to a new tor-independent Mayer--Vietoris square. Corollary~\ref{C:gluing-of-maps-of-alg-spaces} for this square implies that
\[
\AlgSp_{s-\mathrm{fl}}(X')\to 
\AlgSp(X)\times_{\AlgSp(U)} \AlgSp(U')
\]
is fully faithful. Since $f_U$ is flat, the image of
$f_U^*\colon \AlgSp(U)\to \AlgSp(U')$ consists of $s|_{U'}$-flat objects.
Together with Corollary~\ref{C:gluing_special} this gives
\begin{align*}
\AlgSp(X) &\inj
  \AlgSp(X')\times_{\AlgSp(U')} \AlgSp(U) \\
  &= \AlgSp_{s-\mathrm{fl}}(X')\times_{\AlgSp(U')} \AlgSp(U) \\
  &\inj \bigl(\AlgSp(X)\times_{\AlgSp(U)} \AlgSp(U')\bigr)\times_{\AlgSp(U')} \AlgSp(U)
        =\AlgSp(X)
\end{align*}
so $\Phi(s^*W')\cong (W',W_{U'},W_U)$. Thus, $\Phi$ is essentially
surjective for finitely presented algebraic spaces. In fact, by the initial reduction to the 
quasi-compact case, we have proved that $\Phi$ is essentially surjective for triples of 
quasi-separated algebraic spaces.

Let us finally prove that $\Phi$ is also essentially surjective for algebraic
spaces that are not quasi-separated. It is enough to prove that it is
essentially surjective for quasi-compact algebraic spaces. By the previous
argument, it is enough to prove that if $\bar{X}=\varprojlim_\lambda
X_{\lambda}$ is a limit of affine schemes, and
$(\bar{W}',\bar{W}_{U'},\bar{W}_U):=(W',W_{U'},W_U)\times_X \bar{X}$ is in the
essential image of $\Phi$, then so is $(W',W_{U'},W_U)\times_X X_\lambda$ for
sufficiently large $\lambda$.

Thus, let $\bar{W}\to \bar{X}$ be an algebraic space such that
$\Phi(\bar{W})=(\bar{W}',\bar{W}_{U'},\bar{W}_U)$ and pick an affine
presentation $\bar{V}\to \bar{W}$. Note that $\bar{V}\to \bar{W}\to \bar{X}$ is
finitely presented. This induces morphisms of triples
\[
\Phi(\bar{V})=
(\bar{V}',\bar{V}_{U'},\bar{V}_U)\to
(\bar{W}',\bar{W}_{U'},\bar{W}_U)\to
(\bar{X}',\bar{U'},\bar{U})
\]
where the first map is surjective and \'etale and the composition is of finite
presentation. For sufficiently large $\lambda$, we may thus descend this to a
morphism of triples
\[
(V'_\lambda,V_{U',\lambda},V_{U,\lambda})\to
(W'_\lambda,W_{U',\lambda},W_{U,\lambda})\to
(X'_\lambda,U'_\lambda,U_\lambda)
\]
over $X_\lambda$ where the first map is \'etale and the composition is of
finite presentation. Thus, there exists an algebraic space $V_\lambda\to
X_\lambda$, unique up to unique isomorphism, such that $\Phi(V_\lambda)\cong
(V'_\lambda,V_{U',\lambda},V_{U,\lambda})$.

Let $R'_\lambda=V'_\lambda\times_{W'_\lambda} V'_\lambda$ and similarly over
$U'_\lambda$ and $U_\lambda$. Then the triple
$(R'_\lambda,R_{U',\lambda},R_{U,\lambda})$ is locally of finite presentation
and quasi-separated over $(X'_\lambda,U'_\lambda,U_\lambda)$ and hence
isomorphic to $\Phi(R_\lambda)$ for an essentially unique $R_\lambda\to
X_\lambda$. By fully faithfulness, we obtain an \'etale equivalence relation
$\equalizer{R_\lambda}{V_\lambda}$ and we let $W_\lambda$ be its quotient
algebraic space. By fully faithfulness, $\Phi(W_\lambda)$ is isomorphic to
$(W'_\lambda,W_{U',\lambda},W_{U,\lambda})$ and the theorem follows.
\end{proof}
Finally, we prove Theorem \ref{MT:pushout_fmv}.
\begin{proof}[Proof of Theorem \ref{MT:pushout_fmv}]
  We must show that for every algebraic stack $W$, the functor
  \[
  \Phi_{\Hom(-,W)}\colon \Hom(X,W)\to \Hom(X',W)\times_{\Hom(U',W)} \Hom(U,W)
\]
is an equivalence of groupoids. Using Proposition \ref{P:qaff_dom} and the \'etale gluing 
result \cite[Thm.~A]{MR2774654}, we may assume that $X$ is affine and $X'$ is quasi-affine. In particular, we are free to assume that $W$ is quasi-compact. We have already seen that $\Phi_{\Hom(-,W)}$ is fully faithful in
Proposition~\ref{P:cocartesian-in-cat-of-quasi-DM}. To see that it is
essentially surjective, pick a smooth presentation $W_0\to W$ where $W_0$ is an
affine scheme. Pulling back, we obtain a triple in $(X'_0,U'_0,U_0)\in
\AlgSp_\lfp(X')\times_{\AlgSp_{\lfp}(U')} \AlgSp_{\lfp}(U)$; hence, a representable morphism $X_0\to X$ by
Theorem~\ref{MT:glue_fmv}. Since $X'\amalg U'\to X$ is
faithfully flat and quasi-compact, it follows that $X_0\to X$ is smooth.  Also,
\[
\vcenter{\xymatrix{
        U'_0\ar[r]\ar[d] & X'_0\ar[d] \\
        U_0\ar[r] & X_0,\ar@{}[ul]|\square
      }}
\]
is a flat Mayer--Vietoris square.  Since $W_0$ is
affine, we obtain a unique morphism $X_0\to W_0$ compatible with $X'_0\to W_0$
and $U_0\to W_0$ (Proposition~\ref{P:cocartesian-in-cat-of-quasi-DM}). By the full
faithfulness of $\Phi_{\Hom(-,W)}$, the induced morphisms $\equalizer{X_0\times_X X_0}{X_0} \to W_0 \to W$ coincide up to a unique $2$-isomorphism, so there is a unique morphism $X \to W$ and the result follows.
\end{proof}
\bibliography{references}
\bibliographystyle{bibstyle}
\end{document}

%% file: preamble.tex
        \usepackage{latexsym}
        \usepackage{amssymb}
        \usepackage{amsmath}
        \usepackage{amsfonts}
        \usepackage{amsthm}
        \usepackage[hypertexnames=false]{hyperref}
        \ifpdf
          \usepackage[all,knot,poly,2cell]{xy}
        \else
          \usepackage[all,knot,poly,2cell,dvips]{xy}
        \fi
             \UseAllTwocells
        \usepackage{xspace}

        \usepackage{eucal}
        \theoremstyle{plain}
        \newtheorem{theorem}{Theorem}[section]
        \newtheorem{corollary}[theorem]{Corollary}
        \newtheorem{lemma}[theorem]{Lemma}
        
        \newtheorem{proposition}[theorem]{Proposition}
        \newtheorem{maintheorem}{Theorem}

        \theoremstyle{definition}
        \newtheorem{definition}[theorem]{Definition}
        \newtheorem{example}[theorem]{Example}
        \newtheorem{notation}[theorem]{Notation}
        \newtheorem*{example*}{Example}

        \theoremstyle{remark}
        \newtheorem{remark}[theorem]{Remark}
        \newtheorem*{remark*}{Remark}  %%%%% Shortcuts %%%%%
        \newcommand{\suchthat}{\,:\,}
        \newcommand{\itemref}[1]{\eqref{#1}}

        \newcommand{\mathscript}{\mathcal}

        \newcommand{\Z}{\mathbb{Z}}
        \newcommand{\N}{\mathbb{N}}

        \newcommand{\Orb}{\mathcal{O}}   % structure sheaf 
       \DeclareMathOperator{\spec}{Spec} % spectrum of a ring
           \newcommand{\Et}{\mathrm{\acute{E}t}}
           
           \newcommand{\lfp}{\mathrm{lfp}}

           \newcommand{\Aff}{\mathbb{A}}

        \newcommand{\MOD}{\mathsf{Mod}}    % the category of modules      

        \DeclareMathOperator{\coker}{coker}
         \DeclareMathOperator{\Hom}{Hom}
        
        \DeclareMathOperator{\Tor}{Tor}
        
        \newcommand{\COHO}[1]{\mathcal{H}^{{#1}}}
        \newcommand{\trunc}[1]{\tau^{{#1}}}
        \newcommand{\RDERF}{\mathsf{R}}
        \newcommand{\LDERF}{\mathsf{L}}
        
        \newcommand{\DCAT}{\mathsf{D}}

        \newcommand{\QCOH}{\mathsf{QCoh}}

        \renewcommand{\bar}[1]{\overline{{#1}}}

        \newcommand{\tensor}{\otimes}
        
        \newcommand{\homotopic}{\simeq}

        \newcommand{\opp}{\circ}
        
\renewcommand{\subset}{\subseteq}
\numberwithin{equation}{section}

\newcommand{\STACKSABS}{\textbf{Stacks}}
\newcommand{\STACKS}[1]{\STACKSABS/{#1}}

\newcommand{\qcsubscript}{\mathrm{qc}} % subscript

\newcommand{\DQCOH}[1][]{\DCAT_{\qcsubscript{#1}}} 

\newcommand{\QCPBK}[1]{#1^*_{\qcsubscript}}

\newcommand{\spref}[1]{\href{http://stacks.math.columbia.edu/tag/#1}{#1}}

\makeatletter
\newcommand{\labitem}[2]{%
\def\@itemlabel{(\textbf{#1})}
\item
\def\@currentlabel{\textbf{#1}}\label{#2}}
\makeatother

\usepackage{mathtools}
\usepackage{enumitem}
\setlist[enumerate]{font=\normalfont}

\title{Mayer--Vietoris squares in algebraic geometry}
\date{Apr 3, 2023}
\author[J. Hall]{Jack Hall}
\address{School of Mathematics \& Statistics, The University of Melbourne, Parkville, VIC, 3010, Australia}
\email{jack.hall@unimelb.edu.au}
\author[D. Rydh]{David Rydh}
\address{KTH Royal Institute of Technology\\Department of Mathematics\\SE\nobreakdash-100\ 44\ Stockholm\\Sweden}
\email{dary@math.kth.se}
\thanks{The first author was supported by the Australian Research Council DE150101799 while some of this work was completed.}
\thanks{The second author was supported by the Swedish Research Council
2011-5599 and 2015-05554.}
\subjclass[2020]{Primary 14A20; secondary 13F40, 13J10, 14F08, 14F20}

\keywords{Mayer-Vietoris squares, descent, derived categories, algebraic stacks, \'etale morphisms}

\renewcommand{\STACKSABS}{\mathsf{Stacks}}
\newcommand{\CAT}{\mathsf{Cat}}
\renewcommand{\Et}{\mathsf{Et}} % Category of étale morphisms over something
\newcommand{\AlgSp}{\mathsf{AlgSp}} % Category of algebraic spaces over something
\newcommand{\qs}{\mathrm{qs}} % quasi-sep. subscript
\newcommand{\Bl}{\mathrm{Bl}} % Blowup
\newcommand{\AFF}{\mathsf{Aff}} % Affine
\newcommand{\QAFF}{\mathsf{Qaff}}
\newcommand{\inj}{\hookrightarrow}

\newcommand{\clopen}{\mathsf{OC}} % Closed and open
\newcommand{\Cl}{\mathsf{Cl}} % Closed
\newcommand{\Op}{\mathsf{Op}} % Open
\newcommand{\Int}{\mathsf{Int}} % Integral

\newdir{+}{{}*!/-6pt/{}}
\newdir{>+}{@{>}*!/-6pt/{}}
\newcommand{\equalizer}[2]{\xymatrix@1@M=0mm@C=7mm{#1%
 \ar@<.5ex>@{+->+}[r] \ar@<-.5ex>@{+->+}[r] & #2}}

%% file: mayer_vietoris_abstract.tex
We consider various notions of Mayer--Vietoris squares in algebraic geometry. We use these to generalize a number of gluing and pushout results of Moret-Bailly, Ferrand--Raynaud, Joyet and Bhatt. An important intermediate step is Gabber's rigidity theorem for henselian pairs, which our methods give a new proof of.